\documentclass[reqno,tbtags]{amsart}

\usepackage{cite}
\usepackage[shortlabels]{enumitem}
\usepackage{hyperref}
\usepackage{amssymb}

\newtheorem{thm}{Theorem}[section]
\newtheorem{lem}[thm]{Lemma}
\newtheorem{prop}[thm]{Proposition}
\newtheorem{cor}[thm]{Corollary}
\newtheorem*{PC}{Condition (PC)}
\newtheorem{question}[thm]{Question}

\theoremstyle{definition}
\newtheorem{dfn}[thm]{Definition}

\theoremstyle{remark}
\newtheorem{rem}[thm]{Remark}
\newtheorem{exm}[thm]{Example}

\numberwithin{equation}{section}

\newcommand{\dN}{\mathbb{N}}

\newcommand{\dR}{\mathbb{R}}
\newcommand{\dC}{\mathbb{C}}

\newcommand{\cH}{\mathcal{H}}
\newcommand{\cC}{\mathcal{C}}

\newcommand{\cM}{\mathcal{M}}

\DeclareMathOperator{\Lin}{Lin}

\DeclareMathOperator{\tr}{tr}

\newcommand{\dif}{\mathrm{d}}

\DeclareMathOperator{\ran}{rank}

\newcommand{\gk}{\mathfrak{k}}
\newcommand{\gl}{\mathfrak{l}}
\newcommand{\gs}{\mathfrak{s}}

\newcommand{\cB}{\mathcal{B}}

\newcommand{\cE}{\mathcal{E}}

\newcommand{\cN}{\mathcal{N}}

\newcommand{\cX}{\mathcal{X}}
\newcommand{\cW}{\mathcal{W}}
\newcommand{\cV}{\mathcal{V}}

\newcommand{\fX}{\mathfrak{X}}

\newcommand{\fM}{\mathfrak{M}}
\newcommand{\fN}{\mathfrak{N}}

\newcommand{\im}{\mathrm{i}}

\newcommand{\bM}{\mathbf{M}}
\newcommand{\bP}{\mathbf{P}}

\newcommand{\trn}{\mathrm{T}}
\newcommand{\Hq}{\cH_q}
\newcommand{\Hqgg}{\cH_{q,\succeq}}
\DeclareMathOperator{\id}{id}
\DeclareMathOperator{\cone}{cone}
\DeclareMathOperator{\ats}{at}
\DeclareMathOperator{\rre}{Re}

\begin{document}
\title{On the Matricial Truncated Moment Problem. II}

\author{Conrad M\"adler}
\address{University of Leipzig,
Mathematical Institute,
Augustusplatz~10,
04109~Leipzig,
Germany}
\email{maedler@math.uni-leipzig.de}

\author{Konrad Schm\"udgen}
\address{University of Leipzig,
Mathematical Institute,
Augustusplatz~10,
04109~Leipzig,
Germany}
\email{schmuedgen@math.uni-leipzig.de}


\date{\today}

\begin{abstract}
We continue the study of truncated matrix-valued moment problems begun in \cite{ms}.
Let $q\in \dN$.
Suppose that $(\cX,\fX)$ is a measurable space and $\cE$ is a finite-dimensional vector space of measurable mappings of $\cX$ into $\Hq$, the Hermitian $q\times q$~matrices.
A linear functional $\Lambda$ on $\cE$ is called a moment functional if there exists a positive $\Hq$\nobreakdash-valued measure $\mu$ on $(\cX,\fX)$ such that $\Lambda(F)=\int_\cX \langle F, \dif\mu\rangle$ for $F\in \cE$.

In this paper a number of special topics on the truncated matricial moment problem are treated.
We restate a result from \cite{mours} to obtain a matricial version of the flat extension theorem.
Assuming that $\cX$ is a compact space and all elements of $ \cE$ are continuous on $\cX$ we characterize moment functionals in terms of positivity and obtain an ordered maximal mass representing measure for each moment functional.
The set of masses of representing measures at a fixed point and some related sets are studied. 
The class of commutative matrix moment functionals is investigated.
We generalize the apolar scalar product for homogeneous polynomials to the matrix case and apply this to the matricial truncated moment problem.
\end{abstract}

\maketitle

\section{Introduction}
In the present paper we continue the study of general matricial truncated moment problems which was begun in \cite{ms}.
Our aim is to develop a number of important special topics on moment functionals.

First let us recall the general setup from \cite{ms}.
Let $q$ be a positive integer which will be fixed throughout the whole paper.
We denote by $M_q(\dC)$ the complex $q\times q$~matrices, by $\Hq$ the Hermitian matrices of $M_q(\dC)$, and by $\Hqgg $ the non-negative matrices of $\Hq$.
Further, $(\cX,\fX)$ is a measurable space such that $\{x\}\in \fX$ for all $x\in \cX$ and $\cM_q(\cX,\fX)$ denotes the set of positive $\Hq$\nobreakdash-valued measures on $(\cX,\fX)$.
We suppose that $\cE$ is a \textbf{finite-dimensional} real vector space of measurable mappings of $\cX$ into $\Hq$ and $E$ is a \textbf{finite-dimensional} real vector space of real-valued measurable functions on $\cX$.
If not stated explicitly otherwise, these assumptions and notations will be fixed throughout the present paper without mention.

In most of our considerations the matrix entries of elements of $\cE$ are general measurable functions on $(\cX,\fX)$.
Nevertheless, our main interest is on multi-variate polynomials on semi-algebraic sets.
The matricial moment problem in this case is rarely studied in the literature; we mention only a few papers such as \cite{vasilescu}, \cite{bakony}, \cite{cimpric}, \cite{Kimsey}, \cite{KimseyW}, \cite{Le}, \cite{KimseyT}. 

The following definitions are fundamental for this paper.
A linear functional $\Lambda\colon \cE\to \dR$ is called a \emph{moment functional} on $\cE$ if there exists a measure $\mu\in\cM_q(\cX,\fX)$ such that
\[
     \Lambda(F)
     =\int_\cX \langle F(x), \dif\mu\rangle \quad\text{for all }F\in \cE.
\]
In this case, each such $\mu\in\cM_q(\cX,\fX)$ is called a \emph{representing measure} of $\Lambda$.
We denote by $\cM_\Lambda$ the set of all representing measures of $\Lambda$.
Further, a linear mapping $L\colon E\to\Hq$ is said to be a \emph{matrix moment functional} on $E$ if there is a measure $\mu\in\cM_q(\cX,\fX)$ such that
\[
     L(f)
     =\int_\cX f(x) \, \dif\mu\quad\text{for all }f\in E.
\]
In this case, each such $\mu\in\cM_q(\cX,\fX)$ is called a \emph{representing measure} of $L$. 

Let us sketch the contents of this paper.
Detailed explanations and assumptions are given at the beginning of the corresponding sections.

In Section~\ref{prelim}, we collect some notation and facts that will be used later.

In Section~\ref{hankelmatrices}, we define and investigate block Hankel matrices of linear functionals and develop some of their basic properties. 

A fundamental existence result for the scalar truncated moment problem is the flat extension theorem of R.~E.~Curto and L.~A.~Fialkow \cite{cf1}, \cite{cf2}.
General versions of this result were proved in \cite{mours} for non-commutative $\ast$\nobreakdash-algebras and in \cite{KimseyT} for matrices of polynomials in several variables.
In Section~\ref{flatextension}, we adapt a result obtained in \cite{mours} to derive a flat extension theorem in our general setting (Theorem~\ref{flateth}).

In Sections~\ref{existence} and~\ref{orderded}, we assume that $\cX$ is a compact space and all $F\in \cE$ are continuous on $\cX$.
In Section~\ref{existence} we prove an existence result (Theorem~\ref{compactpos}) in terms of the positivity of the corresponding linear functional.
This result is essentially used in Section~\ref{orderded} to obtain ordered maximal mass representing measures (see Definition~\ref{maxmassmeasure}) for each moment functional on $\cE$ (Theorem~\ref{maxmassmeasure}).

In Section~\ref{masses}, we study the set of possible masses at a fixed point among all representing measures of a moment functional. 
In Section~\ref{largestmasses}, we consider the existence of a largest mass at some atom.
We give a criterion for the existence and we develop a counterexample in detail.

Section~\ref{commuative} deals with matrix moment functionals $L$ on $E$ for which the set $L(E)$ is commutative.
We characterize such functionals by various equivalent conditions (Theorem~\ref{cummcond}).

In Section~\ref{transport}, we investigate how moment functionals and their representing measures behave under positive linear mapping of matrices.

In Section~\ref{apolar}, we generalize the apolar scalar product for homogeneous polynomials to the matrix case and apply this to the matricial truncated moment problem (Theorem~\ref{T19.14}).

\section{Preliminaries}\label{prelim}
Let $\tr$ denote the trace and $\langle \cdot,\cdot\rangle$ the scalar product on $M_q(\dC)$ defined by
\begin{align}\label{defscalar}
    \langle X,Y\rangle
    = \tr(XY^*)\quad\text{for all }X,Y\in M_q(\dC).
\end{align}
(Note that $\langle \cdot,\cdot\rangle$ is linear in the first variable and anti-linear in the second.)
For $X,Y\in \Hqgg$ the scalar product is non-negative, since
\begin{equation}\label{scalarpos}
    \langle X,Y\rangle
    = \tr(XY)
    =\tr(\sqrt{Y} X \sqrt{Y})
    \geq 0.
\end{equation} 
For $ j,k\in\{1,\dotsc,q\}$, let $e_{jk}$ denote the standard matrix unit of $M_q(\dC)$. Further, we abbreviate 
\begin{align*}
    H_{jk}&:=\frac{1}{\sqrt{2}}(e_{jk}+e_{kj})\text{ if }j<k,&
    H_{jj}&:= e_{jj},&
    H_{jk}&:= \frac\im{\sqrt{2}}(e_{jk}-e_{kj})\text{ if }j>k. 
\end{align*}
Then the set $\{H_{jk}\colon j,k=1,\dotsc,q\}$ is an orthonormal basis of the real Hilbert space $(\Hq, \langle\cdot, \cdot\rangle )$.
 
We denote by $M_q(E)$ the $q\times q$~matrices with entries from the complexification $E+\im E$ of $E$ and by $\Hq(E)$ the Hermitian matrices in $M_q(E)$. 

Throughout this paper, we often use the one-to-one correspondence between real-linear mappings $L\colon E\to \Hq$ and real-linear functionals $\Lambda$ on $\Hq(E)$ developed in \cite{ms}, see formula~(3.1) and Lemmas~3.3 and~3.4 therein.
The functional $\Lambda$ associated with $L$, denoted $\Lambda_L$, is defined by 
\begin{equation}\label{defcl}
     \Lambda_L(F)
     :=\sum_{j,k=1}^q\langle L( \langle F,H_{jk}\rangle), H_{jk}\rangle,
     \quad F\in \Hq(E).
\end{equation}
The mapping $L$ can be uniquely recovered from the functional $\Lambda_L$ by 
\begin{equation}\label{lfvv}
    v^*L(f)v
    =\Lambda_L(fvv^*)
    \quad \text{for }f\in E,\, v\in\dC^q,
\end{equation} 
or by 
\begin{equation}\label{lfh}
    \Lambda_L(fH)
    =\langle H, L(f)\rangle
    \quad\text{for }f\in E,\, H\in \Hq.
\end{equation}
Recall from \cite[Proposition~4.5]{ms} that $\Lambda_L$ is a moment functional on $\Hq(E)$ if and only if $L$ is a matrix moment functional on $E$ and that in this case the corresponding sets of representing measures coincide.

Since each $A\in M_q(E)$ can be uniquely written as $A=F+\im G$, with $F,G\in \Hq(E)$, and the scalar product \eqref{defscalar} is linear in the first variable, it follows from \eqref{lfh} that $\Lambda_L$ extends to a complex-linear functional on $M_q(E)$, denoted again $\Lambda_L$, such that
\begin{equation}\label{traceaux}
    \Lambda_L(fA)
    =\langle A, L(f)\rangle
    =\tr(AL(f))
    \quad \text{for all }f\in E, \, A\in M_q(\dC).
\end{equation}
Clearly, then each complex-linear functional $\Lambda$ on $M_q(E)$ is of the form $\Lambda_L$, where $L$ is the unique real-linear mapping $L\colon E\to \cH_q$ satisfying \eqref{lfvv} or \eqref{lfh}.

Next we recall some measure-theoretic facts, see \cite{ms} for more details.
Suppose $\mu\in \cM_q(\cX,\fX)$.
Let $\tau:=\tr \mu$ be the trace measure of $\mu$.
Then there exists a matrix $\Phi(x)=(\phi_{ij}(x))_{i,j=1}^q$ of measurable functions $\phi_{ij}$ on $(\cX,\fX)$ such that $\Phi(x)\in \Hqgg $ for all $x\in \cX$ and $\dif\mu=\Phi(x) \dif\tau$.
A measurable mapping $F\colon \cX\to\cH_q$ on $(\cX,\fX)$ is said to be \emph{$\mu$\nobreakdash-integrable} if $\langle F,\Phi\rangle\in L^1(\tau;\dR)$.
In this case, $\int_{\cX} \langle F,\dif\mu\rangle$ is defined by
\[
    \int_{\cX} \langle F,\dif\mu\rangle
    :=\int_{\cX }\langle F,\Phi\rangle \dif\tau.
\]
Let $L^1(\mu;\cH_q)$ denote the set of all $\mu$\nobreakdash-integrable mappings $F\colon \cX\to \cH_q$.
If $\cE\subseteq L^1(\mu;\cH_q)$, then let $\Lambda^\mu\colon\cE\to\dR$ be defined by
\[
    \Lambda^\mu(F)
    :=\int_\cX \langle F,\dif\mu\rangle.
\]
If $E\subseteq L^1(\tau;\dR)$, then let $L^\mu \colon E\to \cH_q$ be defined by
\[
    L^\mu(f)
    :=\int_\cX f \dif\mu.
\]
 
Let $x\in \cX$ and $M\in \Hqgg $.
Then $M\delta_x$ belongs to $\cM_q(\cX,\fX)$ and fulfills $M\delta_x(X)=M$ if $x\in X$ and $M\delta_x(X)=O$ if $x\notin X$ for $X\in \fX$.
Further, $\ell_x$ denotes the point evaluation at $x$, that is, $\ell_x(F)=F(x)$ for $F\in \cE$.

A measure $\nu\in \cM_q(\cX,\fX)$ is called \emph{finitely atomic} if there exists a finite subset $N$ of $\cX$ such that $\nu(\cX\setminus N)=O$. 
Any such measure $\nu$ is of the form $\nu=\sum_{j=1}^k M_j\delta_{x_j}$, where $x_1,\dotsc,x_k\in\cX$ and $M_1,\dotsc,M_k\in\Hqgg $, and we have
$\int_\cX \langle F,\dif\nu\rangle =\sum_{j=1}^k \tr(F(x_j)M_j)$. \smallskip

The following important result, the matricial Richter--Tchakaloff theorem \cite[Theorem~5.1]{ms}, will be often used in the sequel.

\begin{thm}\label{richterm}
Each moment functional $\Lambda$ on $\cE$ has a finitely atomic representing measure $\nu=\sum_{j=1}^k M_j\delta_{x_j} $, where $x_1,\dotsc,x_k\in \cX$, $M_1,\dotsc,M_k\in\Hqgg $ and $k\leq \dim \cE$, that is, we have
\[
     \Lambda (F)
     =\sum_{j=1}^k \tr(F(x_j)M_j)\quad\text{for all }F\in \cE.
\]
\end{thm}

For the next lemma we could not find an explicit reference in the literature; therefore we include a complete proof.
The special case stated as Corollary~\ref{totallycor} is essentially used in Sections~\ref{masses} and~\ref{orderded} below via Lemma~\ref{auxl}.

\begin{lem}\label{totallyord}
Let $(\cH,\langle \cdot,\cdot\rangle)$ be a complex Hilbert space and let $B(\cH)_+$ denote the bounded positive self-adjoint operators on $\cH$ equipped with the usual order relation ``\/$\leq$''.
Suppose that $\cN$ is a non-empty totally ordered set of $B(\cH)_+$ which is bounded from above (that is, there is a constant $c>0$ such that $A\leq c\cdot I$ for all $A\in \cN$).
Then there exist an operator $N_0\in B(\cH)_+$ and a net $(N_j)_{j\in J}$ from $\cN$ such that $A\leq N_0$ for all $A\in \cN$ and $N_0=\lim_j N_j$ in the weak operator topology of $\cH$ (that is, $\langle N_0\varphi,\psi\rangle =\lim_j\langle N_j\varphi,\psi\rangle$ for all $\varphi,\psi\in \cH$).
\end{lem}
\begin{proof}
For $\varphi\in \cH$ we define
\[
    \gs(\varphi)
    :=\sup \{\langle A\varphi,\varphi\rangle\colon A\in \cN \}.
\]
Since $\cN$ is bounded from above and $\cN\subseteq B(\cH)_+$, $\gs(\varphi)$ is finite and we have
\begin{equation}\label{sbounded}
    \gs(\varphi)
    =\lvert\gs(\varphi)\rvert
    \leq c\lVert\varphi\rVert^2
    \quad\textit{for all }\varphi\in \cH.
\end{equation}
Without loss of generality we can assume that $A\geq I$ for all $A\in \cN$.
Then $\varphi\mapsto \langle A\varphi,\varphi\rangle^{1/2}=\lVert A^{1/2}\varphi\rVert$ is a norm, so that $\varphi\mapsto \gs(\varphi)^{1/2}$ is a norm on $\cH$.

For any $A\in\cN$ the norm $\varphi\mapsto\langle A \varphi,\varphi\rangle^{1/2}$ on $\cH$ satisfies the parallelogram identity
\begin{equation}\label{paral}
    \langle A(\varphi+\psi),\varphi+\psi\rangle +\langle A(\varphi-\psi),\varphi-\psi\rangle
    = 2(\langle A\varphi,\varphi\rangle+\langle A\psi,\psi\rangle).
\end{equation}
Fix $\varphi,\psi\in \cH$ and let $\varepsilon >0$.
Since $\cN$ is totally ordered (!), there exists an operator $A_\varepsilon\in \cN$ such that
\begin{gather*}
    \lvert\gs(\varphi\pm\psi)-\langle A_\varepsilon (\varphi\pm\psi),\varphi\pm \psi\rangle\rvert\leq \varepsilon,\\ 
    \lvert\gs(\varphi)-\langle A_\varepsilon \varphi,\varphi\rangle\rvert\leq \varepsilon,\quad\lvert\gs(\psi)-\langle A_\varepsilon \psi,\psi\rangle\rvert\leq \varepsilon.
\end{gather*} 
Since $\varepsilon>0$ is arbitrary, it follows from \eqref{paral} that the norm $\varphi\mapsto \gs(\varphi)^{1/2}$ satisfies the parallelogram identity as well.
Therefore, by a classical result of J.~v.~Neumann (see e.\,g.\ \cite[Proposition~1.6]{weid}), this norm comes from a scalar product, denoted $\gs(\cdot,\cdot)$, which is given by the polarization identity
\begin{equation}\label{polarid}
    4\gs (\varphi,\psi)
    = \gs( \varphi+\psi) - \gs(\varphi-\psi) + \im \gs(\varphi+\im \psi) - \im \gs(\varphi-\im \psi).
\end{equation}
Then $\gs(\varphi,\varphi)=\gs(\varphi)$ for $\varphi\in \cH$.
Thus, by \eqref{sbounded}, $\gs(\cdot,\cdot)$ is a bounded positive sesquilinear form on $\cH$.
Hence there exists an operator $N_0\in B(\cH)_+$ such that $\gs(\varphi,\psi)=\langle N_0\varphi,\psi\rangle $ for all $\varphi,\psi\in \cH$.

Now fix $\varphi,\psi\in \cH$.
Using once more that $\cN$ is totally ordered, for any $\varepsilon>0$ we can find an operator $N_\varepsilon\in \cN$ such that
\begin{align*}
    \lvert\gs(\varphi\pm\psi)-\langle N_\varepsilon (\varphi\pm\psi),\varphi\pm \psi\rangle\rvert&\leq \varepsilon&
    &\text{and}&
    \lvert\gs(\varphi\pm\im\psi)-\langle N_\varepsilon (\varphi\pm \im \psi,\varphi\pm\im \psi\rangle\rvert&\leq \varepsilon.
\end{align*}
Therefore, by \eqref{polarid} and the polarization identity for the operator $N_\varepsilon$ this implies that 
\[
    \lvert\langle N_0\varphi,\psi\rangle-\langle N_\varepsilon\varphi,\psi\rangle\rvert
    =\lvert\gs(\varphi,\psi)-\langle N_\varepsilon\varphi,\psi\rangle\rvert
    \leq \varepsilon.
\]
From this it follows that $N_0$ is the limit in the weak operator topology of a net from $\cN$.
\end{proof}

\begin{cor}\label{totallycor}
Let $\cN$ be a subset of $\Hqgg $ which is totally ordered and bounded from above with respect the order relation ``\/$\succeq$''.
Then there exist a matrix $N_0\in\Hqgg $ and a sequence $(N_n)_{n\in \dN}$ of elements of $\cN$ such that $N_0\succeq N$ for all $N\in \cN$ and $N_0=\lim_n N_n$ in the norm topology of $\Hq$. 
\end{cor} 
\begin{proof}
The assertion follows from Lemma~\ref{totallyord}, applied to the matrices of $\cN$ acting as operators on the Hilbert space $\cH=\dC^q$.
Then the order relation ``$\succeq$'' corresponds to ``$\geq$''.
In this case the weak operator topology is given by a norm, so we can choose a sequence from $\cN$ rather than a net. 
\end{proof}

\section{Hankel matrices}\label{hankelmatrices}
Let $n\in \dN$.
We denote by $E_n$ the real vector space spanned by all products $g_1\dotsm g_n$, where $g_j\in E$ or $g_j=1$ for $j=1,\dotsc,n$.
Further, let $M_q(E_n)$ be the $q\times q$~matrices with entries from the complex vector space $E_n+\im E_n$.
Then $M_q(E_n)$ is a complex vector space equipped with an involution given by $(a_{jk})^*:= (\,\overline{a_{kj}}\,)$ for $(a_{jk})\in M_q(E_n)$.
Since $1\in E_n$, $M_q(E_n)$ contains the unit matrix.
Clearly, we have
\[
    M_q(E_k) \cdot M_q(E_\ell)
    \subseteq M_q(E_{k+\ell})
    \quad\text{for all }k,\ell\in \dN.
\]

Now we fix a basis $\{f_1,\dotsc,f_d\}$ of the real vector space $E+\dR\cdot 1$.
For $\gk=(k_1,\dotsc,k_d)\in \dN_0^d$ we set $f^\gk:=f_1^{k_1}\dotsm f_d^{k_d}$ and $\lvert\gk\rvert :=k_1+\dotsb+k_d$.
Further, we abbreviate $I(d,n):=\{ \gk\in \dN_0^d\colon \lvert\gk\rvert \leq n\}$ for $n\in \dN$.

Clearly, $E_n$ is the complex linear span of the set $\{f^\gk\colon \gk\in I(d,n)\}$.
Note that all functions $f^\gk$ on $\cX$ are real-valued.

Each matrix $A\in M_q(E_n)$ can be written in the form 
\begin{align*}
    A
    =\sum_{\gk\in I(d,n)} f^\gk A_\gk,
    \quad\text{where }A_\gk\in M_q(\dC)\text{ for }\gk\in I(d,n).
\end{align*}
Now we fix an ordering of the index set $I(d,n)$ (for instance, the lexicographic ordering) and arrange the coefficients of $A\in M_q(E_n)$ as a column vector, denoted by $\vec A=(A_\gk)$, according to this ordering.

\begin{dfn}\label{D0912}
Let $\Lambda$ be a complex-linear functional on $M_q(E_{2n})$.
The \emph{Hankel matrix} $H_n(\Lambda)$ of $\Lambda$ is the $\binom{d+n}{n}\times\binom{d+n}{n}$\nobreakdash~block matrix defined by
\[
    H_n(\Lambda)
    :=(h_{\gk\gl}:=L(f^{\gk+\gl}))_{\gk,\gl\in I(d,n)},
\]
where $L\colon E_{2n}\to \Hq$ is the unique real-linear mapping given by \eqref{lfvv}, with $\Lambda_L$ replaced by $\rre\Lambda$ and $E$ replaced by $E_{2n}$.
\end{dfn}

Note that (in contrast to other places of the literature, see e.\,g.\ \cite{KimseyT}) the entries $h_{\gk\gl}=L(f^{\gk+\gl})$ of this Hankel matrix $H_n(\Lambda)$ are matrices of $\Hq$.
We consider the block matrix $H_n(\Lambda)$ written as a matrix with respect to the fixed ordering of the index set $I(d,n)$.
Then a product such as $\vec {B}^*H_n(\Lambda) \vec{A}$, with $A,B\in M_q(E_n)$, is well defined and gives a matrix of $M_q(\dC)$.

The following two lemmas contain matrix-valued counterparts of the corresponding assertions in the scalar case, see e.\,g.\ \cite[Proposition~17.17]{sch17}.

\begin{lem}
For $ A=\sum_{\gl\in I(d,n)} f^\gl A_\gl, B=\sum_{\gk\in I(d,n)} f^\gk B_\gk \in M_q(E_{n})$, we have
\begin{equation}\label{lambdahankel}
    \Lambda(AB^*)
    =\tr(\vec {B}^*H_n(\Lambda)\vec{A}).
\end{equation}
\end{lem}
\begin{proof}
Applying formula \eqref{traceaux} we compute
\[\begin{split}
    &\Lambda(AB^*)
    =\Lambda\left(\sum_{\gl,\gk\in I(d,n)}(f^\gl A_\gl)(f^\gk B_\gk)^*\right)
    = \sum_{\gl,\gk\in I(d,n)} \Lambda(f^{\gl+\gk}A_\gl B_\gk^*)\\
    &=\sum_{\gl,\gk\in I(d,n)} \tr(A_\gl B_\gk^* L(f^{\gl+\gk})) 
    =\tr \left(\sum_{\gk,\gl\in I(d,n)}B_\gk^* L(f^{\gk+\gl})A_\gl\right)
    =\tr(\vec {B}^*H_n(\Lambda) \vec{A}).\qedhere
\end{split}\]
\end{proof}

\begin{lem}\label{poshankel}
$H_n(\Lambda)\succeq O$ if and only if $\Lambda(A^*A)\geq 0$ for all $A\in M_q(E_n)$.
\end{lem}
\begin{proof} 
Let $(\cdot,\cdot)$ denote the standard scalar product of $\dC^q$.
Let $A\in M_q(E_n)$.
Then $B:=A^*$ belongs to $M_q(E_n)$.
We consider the matrix $\vec{B}^* H_n(\Lambda) \vec{B}$ acting on $\dC^q$.
For $v\in \dC^q$, we compute
\begin{equation}\label{HLambda0}
\begin{split}
     (\vec{B}^*H_n(\Lambda) \vec{B}v,v)
     &=\sum_{\gk,\gl\in I(d,n)} (B_\gk^* L(f^{\gk+\gl})B_\gl v,v)\\
     &=\sum_{\gk,\gl\in I(d,n)} (L(f^{\gk+\gl})B_\gl v,B_\gk v)
     =\sum_{\gk,\gl\in I(d,n)} (h_{\gk\gl}B_\gl v,B_\gk v).
\end{split}
\end{equation}
From \eqref{HLambda0} we see that if $H_n(\Lambda)\succeq O$, then $\vec {B}^* H_n(\Lambda) \vec{B}\succeq O$ and therefore $\Lambda(A^*A)=\Lambda(BB^*)=\tr(\vec {B}^*H_n(\Lambda) \vec{B})\geq 0$ by \eqref{lambdahankel}.

To prove the converse direction we first note that the Hankel matrix $H_n(\Lambda)$ acts on the direct sum of $\binom{d+n}{n}$ copies of $\dC^q$.
Let $u_\gk$, where $ \gk\in I(d,n)$, be arbitrary vectors of $\dC^q$.
Further, let $\{v_1,\dotsc,v_q\}$ be an orthonormal basis of $\dC^q$.
For each $\gk\in I(d,n)$ we choose a matrix $A_\gk\in M_q(\dC)$ such that 
\begin{align}\label{agk}
     A_\gk^* v_1&=u_\gk\quad
     \text{and}\quad 
     A_\gk^* v_j=0\quad\text{for }j\geq 2.
\end{align} 
Then, using \eqref{lambdahankel}, the definition of the trace, \eqref{HLambda0} and finally \eqref{agk} we derive
\[\begin{split}
    \Lambda(A^*A)
    &=\Lambda(BB^*)
    =\tr(\vec {B}^*H_n(\Lambda) \vec{B})
    =\sum_{j=1}^q (\vec{B}^*H_n(\Lambda) \vec{B}v_j,v_j)\\
    &=\sum_{j=1}^q\sum_{\gk,\gl\in I(d,n)} (h_{\gk\gl}B_\gl v_j,B_\gk v_j)
    =\sum_{j=1}^q\sum_{\gk,\gl\in I(d,n)} (h_{\gk\gl}A_\gl^* v_j,A_\gk^* v_j)\\
    &=\sum_{\gk,\gl\in I(d,n)} (h_{\gk\gl}A_\gl^* v_1,A_\gk^* v_1)
    =\sum_{\gk,\gl\in I(d,n)} (h_{\gk\gl}u_\gl,u_\gk).
\end{split}\]
Therefore, if $\Lambda(A^*A)\geq 0$ for all $A\in M_q(E_n)$, then $H_n(\Lambda)\succeq O$.
\end{proof}

\section{The matricial flat extension theorem}\label{flatextension}
First we recall some basic definitions and facts from \cite[Section~2]{mours}.
Let $m\in \dN$.

\begin{dfn}\label{defflatfunc}
A linear functional $\Lambda$ on $M_q(E_{2m+2})$ is called \emph{positive} if
\[
    \Lambda(A^*A)
    \geq 0
    \quad\text{for all }A\in M_q(E_{m+1}).
\]
\end{dfn}

Lemma~\ref{poshankel} gives a reformulation of the positivity of the functional $\Lambda$ in terms of the Hankel matrix:
\emph{$\Lambda$ is positive if and only if $H_{m+1}(\Lambda)\succeq O$.} 

\begin{dfn}\label{D0934}
A linear functional $\Lambda$ on $M_q(E_{2m+2})$ is said to be a \emph{flat extension} with respect to $M_q(E_m)$ if
\[
    M_q(E_{m+1})
    =M_q(E_{m}) + \ker_{m+1} \Lambda,
\]
where 
\[
    \ker_{m+1} \Lambda\colon
    = \left\{A \in M_q(E_{m+1})\colon \Lambda (B^*A) = 0\text{ for all }B\in M_q(E_{m+1}) \right\}.
\]
\end{dfn}
 
If the functional $\Lambda$ is positive, then by \cite[Lemma~2.2]{mours} we have
\[
    \ker_{m+1} \Lambda
    =\left\{A \in M_q(E_{m+1}) \colon \Lambda (A^* A) = 0\right\}.
\]
 
As usual, flat extensions can be characterized in terms of Hankel matrices:
\emph{A linear functional $\Lambda$ on $M_q(E_{2m+2})$ is a flat extension with respect to $M_q(E_m)$ if and only if}
\[
    \ran H_{m+1}(\Lambda)=\ran H_{m}(\Lambda).
\]
This follows by combining Proposition~2.1 and Definition~2.3 in \cite{mours}, applied to $\cC:=M_q(E_{m+1})$ and $\cB:=M_q(E_{m})$ in the notation used therein.

The following theorem is a matricial version of the flat extension theorem.

\begin{thm}\label{flateth}
Let $m\in \dN$ and let $\Lambda$ be a linear functional on $M_q(E_{2m+2})$.
Suppose $\Lambda$ is positive (or equivalently, $H_{m+1}(\Lambda)\succeq O$) and $\Lambda$ is a flat extension with respect to $M_q(E_m)$ (or equivalently, $\ran  H_{m+1}(\Lambda)=\ran  H_{m}(\Lambda))$.
Then there exist points $x_1,\dotsc,x_n\in \cX$ and matrices $M_1,\dotsc,M_n\in \Hqgg $ such that
\[
    \Lambda(A)
    =\sum_{i=1}^n \tr(A(x_i)M_i)
    \quad\text{for all }A \in M_q(E_{2m+2}).
\]
That is, $\Lambda$ is a moment functional on $M_q(E_{2m+2})$ with representing measure $\nu=\sum_{j=1}^n M_j\delta_{x_j}$.
\end{thm}
\begin{proof}
This result is an application of the general flat extension theorem \cite[Theorem~3.1]{mours} to the present setup.
The proof of Theorem~\ref{flateth} follows almost verbatim the same pattern as the proof of Proposition~4.3 in \cite{mours} does in the special case $F=\Lin \{x_1,\dotsc,x_d\}$ and $\cX=\dR^d$.
We do not carry out the details and refer to \cite{mours}.
At the end of this proof it was shown that there exist points $x_i\in \cX$ and vectors $u_i=(u_{ji})\in \dC^d$, $i=1,\dotsc,n$, such that 
\begin{equation}\label{trunvmatrices}
    \Lambda((a_{jk}))
    =\sum_{j,k=1}^q\sum_{i=1}^n a_{jk}(x_i)u_{ki}\overline{u_{ji}}
\end{equation}
for $A=(a_{jk})\in M_q(E_{2m+2})$.
Let $M_i$ denote the matrix with $(k,j)$\nobreakdash-entry $u_{ki}\overline{u_{ji}}$.
Clearly, then $M_i\in \Hqgg $ and \eqref{trunvmatrices} read as $\Lambda(A)=\sum_{i=1}^n  \tr( A(x_i)M_i)$.
This completes the proof of the theorem.
\end{proof}

\section{An existence theorem}\label{existence}
For a moment functional $\Lambda$ on $\cE$ it is obvious that $\Lambda(F)\geq 0$ for $F\in \cE_\succeq$, where $\cE_\succeq:=\{A\in \cE\colon A\succeq O\}$.
It is natural to ask whether this positivity condition is sufficient for being a moment functional.
For the \emph{full} matricial moment problem of polynomials on $\dR^d$ this was proved in \cite[Proposition~2.1]{sch87}.
For the truncated moment problem it fails already in the scalar case (see \cite[Example~9.17]{sch17}).
However, if $\cX$ is a \emph{compact} space and all functions of $\cE$ are \emph{continuous}, this condition is sufficient, as shown by the next theorem.

\begin{thm}\label{compactpos}
Suppose that $\cX$ is a compact topological space and all $F\in \cE$ are continuous on $\cX$.
Suppose there is a matrix $e\in \cE$ and an $\varepsilon >0$ such that $e(x)\succeq \varepsilon I_q$ for all $x\in \cX$.
If $\Lambda$ is a linear functional on $\cE$ such that $\Lambda(F)\geq 0$ for all $F\in \cE_{\succeq}$, then $\Lambda$ is a moment functional.
\end{thm} 
\begin{proof}
The proof follows the lines of the proof of Proposition~2.1 in \cite{sch87}, adapted to the present situation.

We denote by $M_q(C(\cX))$ the $q\times q$~matrices with entries from $C(\cX)$, by $\Hq(C(\cX))$ the subset of Hermitian matrices 
and by $\Hqgg (C(\cX))$ the matrices $A \in \Hq(C(\cX))$ such that $A(x)\succeq O$ for all $x\in \cX$.
Since all $F\in \cE$ are continuous, $\cE\subseteq \cH_q(C(\cX))$.
Because $\cX$ is compact, for each $A\in \Hq(C(\cX))$ there exists $C>0$ such that $A(x)\preceq C\cdot I_q$, hence $A(x)\preceq C\varepsilon^{-1}\cdot e(x)$ for all $x\in \cX$.
Since $e\in \cE$, this implies that the vector space $\cE$ is cofinal in $\Hq(C(\cX))$ with respect to the matrix ordering ``$\succeq$''. 
By assumption the linear functional $\Lambda$ is non-negative on $\cE_\succeq$.
Therefore, $\Lambda$ can be extended to a real-valued linear functional on $\Hq(C(\cX))$ and then to a complex-valued linear functional, denoted again $\Lambda$, on $M_q(C(\cX))$ such that
\begin{equation}\label{succeqpos}
    \Lambda(A)
    \geq 0
    \quad \text{for all }A\in \Hqgg (C(\cX)).
\end{equation}

For $f\in C(\cX)$ we define $l_{kn}(f)=\Lambda(fe_{kn})$, $k,n=1,\dotsc,q$, and 
\[
    T(f)
    :=\sum_{k=1}^q l_{kk}(f).
\]
As usual, let $C(\cX)_+:=\{f\in C(\cX)\colon f(x)\geq 0\text{ for all }x\in \cX\}$.
For $f\in C(\cX)_+$ and $k=1,\dotsc,q$, we set $A_{f,k}:=(\sqrt{f}e_k)^*(\sqrt{f}e_k)$.
Then $A_{f,k }\in \Hqgg (C(\cX))$ and $\Lambda(A_{f,k})=l_{kk}(f)\geq 0$ by \eqref{succeqpos}.
Therefore, $T(f)\geq 0$ for all $f\in C(\cX)_+$, so $T$ is a positive linear functional on $C(\cX)$.
Since $\cX$ is compact, there exists a Radon measure $\tau$ on $\cX$ such that
\[
    T(f)
    =\int_\cX f(x)\, \dif{\tau}(x)
    \quad\text{for all }f\in C(\cX).
\]

Let $f\in C(\cX)_+$.
For $t=(t_1,\dotsc, t_q)\in \dC^q$, we consider the matrix $B_{f,t}:=(\,ft_i\overline{t_j}\,)\in \Hqgg(C(\cX))$.
From \eqref{succeqpos} we obtain
\begin{equation}\label{postmtn}
    \Lambda(B_{f,t})
    = \sum_{k,n=1}^q l_{kn}(f)t_k\overline{t_n}
    \geq 0.
\end{equation}
Since the $t_j\in \dC$ are arbitrary, it follows that $l_{kk}(f)\geq 0$, $l_{nn}(f)\geq 0$ and
\[
    \lvert l_{kn}(f)\rvert
    \leq \sqrt{l_{kk}(f)}\, \sqrt{l_{nn}(f)}
    \leq \frac{1}{2} \left( l_{kk}(f)+l_{nn}(f)\right)
    \leq T(f)
    = \int_\cX f\, \dif{\tau}
\]
for $k,n=1,\dotsc,q$.
This implies that
\[
    \lvert l_{kn}(f)\rvert
    \leq 4\int_\cX \lvert f(x)\rvert \, \dif{\tau}(x)
    \quad \text{for all }f \in C(\cX).
\]
Hence there exists a function $\phi_{kn}\in L^\infty(\cX;\tau)$ such that 
\begin{equation}\label{lkngkn}
    l_{kn}(f)
    = \int_\cX f(x)\phi_{kn}(x)\, \dif{\tau}(x)
    \quad \text{for all } f\in C(\cX).
\end{equation}
Now we use the matrix $B_{f,t}=(\,ft_i\overline{t_j}\,)_{i,j=1}^q$ of $\Hqgg(C(\cX))$ once more.
By \eqref{postmtn}, 
\begin{align*}
    \Lambda(B_{f,t})
    =\sum_{k,n=1}^q l_{kn}(f)  t_k\overline{t_n}
    = \int_\cX f(x)\left[\sum_{k,n=1}^q \phi_{kn}(x)t_k\overline{t_n}\right]\dif\tau(x)
    \geq 0
\end{align*}
for all $f\in C(\cX)_+$.
In particular, $\int_\cX[\sum_{k,n=1}^q \phi_{kn}(x)t_k\overline{t_n}]\dif\tau(x)\geq 0$.
Hence there is a $\tau$\nobreakdash-null set $N_t$ such that  $\sum_{k,n=1}^q \phi_{kn}(x)t_k\overline{t_n}\geq 0$ for all $x\in \cX\setminus N_t$.
We choose a countable dense subset $T$ of $\dC^q$.
Then the set $N:=\bigcup_{t\in T} N_t $ has $\tau$\nobreakdash-measure zero and we have $\sum_{k,n=1}^q \phi_{kn}(x)t_k\overline{t_n}\geq 0$ for all $x\in \cX\setminus N$ and $t\in T$.
Since $T$ is dense in $\dC^q$, the latter holds for all $t\in \dC^q$ and $x\in \cX\setminus N $.
Therefore we conclude that $\Phi(x):=(\phi_{kn}(x))\succeq O$ $\tau$\nobreakdash-a.\,e.\ on $\cX$.
Hence there is a measure $\mu\in \cM_q(\cX)$ given by $\dif\mu(x)=\Phi(x)\dif\tau(x)$.
Let $F=(f_{ij})\in \cE$.
Then, by \eqref{lkngkn}, we derive
\[
    \Lambda (F)
    =\sum_{k,n=1}^q l_{kn}(f_{kn})
    =\sum_{k,n=1}^q \int_\cX f_{kn} \phi_{kn}\, \dif\tau
    =\int_\cX\tr(F\Phi)\, \dif\tau
    =\Lambda^\mu(F),
\]
that is, $\Lambda=\Lambda^\mu$, so $\Lambda$ is a moment functional on $\cE$.
\end{proof}

\begin{cor}\label{C1002}
Retain the assumptions of Theorem~\ref{compactpos}.
For a linear functional $\Lambda$ on $\cE$ the following statements are equivalent:
\begin{enumerate}[(i)]
    \item\label{C1002.i} $\Lambda$ is a moment functional.
    \item\label{C1002.ii} $\Lambda(F)\geq 0$ for all $F\in \cE_{\succeq}$.
\end{enumerate}
\end{cor}
\begin{proof}
The easy part~\ref{C1002.i}$\to$\ref{C1002.ii} holds by \cite[Lemma~4.4]{ms} and the implication~\ref{C1002.ii}$\to$\ref{C1002.i} is Theorem~\ref{compactpos}.
\end{proof}

\section{Masses}\label{masses}
The set of finitely atomic representing measures is, of course, the most important characteristic of a moment functional.
Except from simple cases, it seems to be hopeless to obtain a complete description.
But it is natural to look for the sets of possible atoms and masses.
The set of atoms has been characterized by the core set in \cite[Theorem~8.16]{ms}. 
In this section and the following two we are looking at the set of possible masses at a fixed point.

We begin with some basic definitions.

\begin{dfn}\label{D1005}
For a matrix measure $\mu\in \cM_q(\cX,\fX)$ and a point $x_0\in \cX$, the matrix $\mu(\{x_0\})\in \Hqgg $ is called the \emph{mass} of $\mu$ at $x_0$. 
\end{dfn}

\begin{dfn}\label{D2116}
For a moment functional $\Lambda$ on $\cE$ and $x_0\in \cX$, we denote by
\begin{itemize}
    \item $\bM(\Lambda;x_0)$ the set of masses at $x_0$ of all representing measures of $\Lambda$,
    \item $\bM_\mathrm{max}(\Lambda;x_0)$ the set of maximal elements of $\bM(\Lambda;x_0)$ with respect to ``$\succeq$'',
    \item $\bP(\Lambda;x_0)$ the set of matrices $Y\in \Hqgg$ for which there exists a matrix $X\in \bM(\Lambda;x_0)$ such that $X\succeq Y$.
\end{itemize}
\end{dfn}

Let $\cM$ be a subset of $\Hq$.
Following \cite{pukelsheim}, the \emph{penumbra} of $\cM$ is the set
of matrices $A\in \Hq$ for which there exists a matrix $M\in \cM$ such that $A\preceq M$.
Then, by this definition, $\bP(\Lambda;x_0)$ is the intersection of $\Hqgg$ with the penumbra of $\bM(\Lambda;x_0)$.
Obviously,
\[
    \bM_\mathrm{max}(\Lambda,x_0)
    \subseteq \bM(\Lambda;x_0)
    \subseteq \bP(\Lambda;x_0).
\]

From the matricial Richter--Tchakaloff theorem it follows that $\bM(\Lambda;x_0)$ is equal to the set of masses at $x_0$ of all \emph{finitely atomic} representing measures of $\Lambda$, see \cite[Lemma~7.2]{ms}.

Throughout this section, we assume that $\Lambda$ is a \textbf{moment functional} on $\cE$ and that \textbf{there exist $e\in \cE$ and $\varepsilon >0$ such that $e(x)\succeq \varepsilon I_q$ for all $x\in \cX$.}

\begin{lem}\label{L0934}
If $M\in\bM(\Lambda;x_0)$, then $\Lambda(F)\geq\tr(F(x_0)M)$ for all $F\in\cE_\succeq$ and, in particular, $\Lambda(e)\geq\varepsilon\tr M$.
\end{lem}
\begin{proof}
Let $F\in\cE_\succeq$.
There exists $\mu\in\cM_\Lambda$ with $\mu(\{x_0\})=M$.
Since $F(x)\succeq O$ by assumption and $\Phi(x)\succeq O$ by our convention, $\langle F(x),\Phi(x)\rangle\geq0$ by \eqref{scalarpos} for all $x\in\cX$.
We derive
\[\begin{split}
    \Lambda(F)
    &=\int_\cX\langle F,\dif\mu\rangle
    =\int_\cX\langle F,\Phi\rangle\dif\tau
    \geq\langle F(x_0),\Phi(x_0)\rangle\tau(\{x_0\})\\
    &=\langle F(x_0),\Phi(x_0)\tau(\{x_0\})\rangle
    =\langle F(x_0),\mu(\{x_0\})\rangle
    =\langle F(x_0),M\rangle
    =\tr(F(x_0)M).
\end{split}\]
Since obviously $e$ belongs to $\cE_\succeq$, we can conclude $\Lambda(e)\geq \tr(e(x_0)M)\geq \varepsilon \tr M$.
\end{proof}

\begin{lem}\label{convexbounded}
For any $x_0\in\cX$, $\bM(\Lambda;x_0)$ and $\bP(\Lambda;x_0)$ are bounded convex sets in $\Hq$.
\end{lem}
\begin{proof}
The convexity of $\bM(\Lambda; x_0)$ is clear, since any convex combination of matrices of $\bM(\Lambda; x_0)$ is the mass at $x_0$ of the corresponding convex combination of representing measures.
The convexity of $\bP(\Lambda; x_0)$ is obvious.

We prove that $\bM(\Lambda;x_0)$ is bounded.
Let $X\in\bM(\Lambda; x_0)$.
Because of Lemma~\ref{L0934}, then $\Lambda(e)\geq \varepsilon \tr X$, so that $0\leq \tr X\leq \varepsilon^{-1} \Lambda(e)$. 
Hence $\bM(\Lambda; x_0)$ is a bounded subset in the trace norm and therefore in any norm on $\Hq$. 

By definition, for $Y\in \bP(\Lambda;x_0)$ there exists a matrix $X\in \bM(\Lambda;x_0)$ such that $X\succeq Y\succeq O$.
Hence $\bP(\Lambda;x_0)$ is also bounded.
\end{proof}

\begin{prop}\label{P1014}
Suppose that $\cX$ is a compact topological space and each $F\in \cE$ is continuous on $\cX$.
Let $x_0\in \cX$.
Then $\bP(\Lambda;x_0)$ is a convex compact subset of $\Hqgg $.
\end{prop}
\begin{proof}
That $\bP(\Lambda;x_0)$ is convex and bounded was already stated in Lemma~\ref{convexbounded}.
Thus it remains to prove that the set $\bP(\Lambda; x_0)$ is closed in $\Hq$.
For we suppose that $(Y_n)_{n\in \dN}$ is a sequence of $\bP(\Lambda; x_0)$ converging to some matrix $Y\in \Hq$.
Let $n\in \dN$.
By the definition of $ \bP(\Lambda; x_0)$, there exists a matrix $X_n\in\bM(\Lambda; x_0)$ such that $X_n\succeq Y_n\succeq O$ and there is a representing measure $\mu_n$ for $\Lambda$ such that $\mu_n(\{x_0\})=X_n$.
Then $\tilde{\mu}_n:=\mu_n-X_n\delta_{x_0}$ is a measure of $\cM_q(\cX,\fX)$.
Let $\tilde{\Lambda}_n$ be the moment functional of $\tilde{\mu}_n$ on $\cE$.
By Theorem~\ref{richterm}, $\tilde{\Lambda}_n$ has a representing measure of the form $\nu_n=\sum_{j=1}^{k_n} M_{n,j}\delta_{x_{n,j}}$, where $k_n\leq d:=\dim \cE$, $x_{n,j}\in \cX$ and $M_{n,j}\in \Hqgg $.
Without loss of generality, we can assume $k_n=d$.
Then 
\[\begin{split}
    \Lambda(e)
    &=\tilde{\Lambda}_n(e)+ \tr(e(x_0)X_n)
    =\sum_{i=1}^d \tr(e(x_{n,i})M_{n,i})+\tr(e(x_0)X_n)\\
    &\geq \tr(e(x_{n,j})M_{n,j})+\tr(e(x_0)X_n)
    \geq \varepsilon(\tr M_{n,j} +\tr X_n).
\end{split}\]
Hence the set of all matrices $X_n$ and $M_{n,j}$, where $n\in \dN$ and $j=1,\dotsc,d$, is bounded in $\Hq$ with respect to the trace norm and so in any norm.
Therefore, since $\cX$ is compact, by passing to appropriate sub-sequences if necessary, we can assume that the sequences $(X_n)_{n\in \dN}$ and $(M_{n,j})_{n\in \dN}$ in $\Hq$ and $(x_{n,j})_{n\in \dN}$ in $\cX$ converge for $j=1,\dotsc,d$.
Let $X:=\lim_n X_n$, $M_j:=\lim_n M_{n,j}$ and $x_j:=\lim_n x_{n,j}$.
Then $X\in \Hqgg $ and $M_j\in \Hqgg $.
For $F\in \cE$, 
\begin{equation}\label{lambda}
    \Lambda(F)
    =\tilde{\Lambda}_n(F)+ \tr(F(x_0)X_n)
    =\sum_{j=1}^d \tr(F(x_{n,j})M_{n,j})+ \tr(F(x_0)X_n).
\end{equation}
By assumption each $F\in \cE$ is continuous on $\cX$.
Therefore, we can pass to the limit in \eqref{lambda} and obtain
\[
    \Lambda(F)
    =\sum_{j=1}^d \tr(F(x_{j})M_{j})+ \tr( F(x_0)X).
\]
Hence $\nu:=\sum_{j=1}^d M_j\delta_{x_j} + X\delta_{x_0}$ is a representing measure of $\Lambda$.
Then $\nu(\{x_0\})\succeq X$.
From $X_n\succeq Y_n\succeq O$ we get $X\succeq Y\succeq O$, so that $\nu(\{x_0\})\succeq Y\succeq O$.
This means that $Y$ belongs to $\bP(\Lambda;x_0)$.
\end{proof}

For our considerations the following two  sets are useful.

\begin{dfn}\label{D0709}
Let $\Lambda$ be a moment functional on $\cE$,  $x_0\in \cX$, and  $M_0\in\Hqgg$.
We denote by
\begin{itemize}
    \item $\fM(\Lambda;x_0,M_0)$ the set of $M\in \Hqgg$ satisfying $M\succeq M_0$ and $\Lambda(F) \geq \tr(F(x_0)M)$ for all $F\in \cE_\succeq$.
    \item $\fM_\mathrm{max}(\Lambda;x_0,M_0)$ the set of  maximal elements of $\fM(\Lambda;x_0,M_0)$ with respect to ``$\succeq$''.
\end{itemize}
\end{dfn}

Clearly
\begin{equation}\label{G0814}
  \fM(\Lambda;x_0,M_1)
  \subseteq\fM(\Lambda;x_0,M_0)
  \subseteq\fM(\Lambda;x_0,O)
  \quad\text{for }M_1\succeq M_0\succeq O.    
\end{equation}

The following lemma is a step in the proof of Proposition~\ref{maxdominating}.
The second assertion is used in the proof of Theorem~\ref{constrmethod} below.

\begin{lem}\label{auxl}
Let $x_0\in \cX$ and $M_0\in \Hqgg $.
If~ $\fM(\Lambda;x_0,M_0)\neq\emptyset$, then $\fM(\Lambda;x_0,M_0)$ has at least one maximal element with respect to the ordering ``\/$\succeq$''. 
In particular, $\fM(\Lambda;x_0,O)$ has at least one maximal element.

Moreover, for each maximal element $M_1$ of $\fM(\Lambda;x_0,O)$ with $M_1\neq O$, there exists a sequence $(F_n)_{n\in \dN}$ of $\cE_\succeq$ such that 
\begin{equation}\label{auxsequ}
    \lim_{n\to\infty} \Lambda(F_n)=1\text{ and }\tr(F_n(x_0)M_1)=1\text{ for all }n\in \dN.
\end{equation}
\end{lem}
\begin{proof}
To prove the first assertion we want to apply Zorn's lemma.
First we note that $\fM(\Lambda;x_0,M_0)$ is not empty by assumption.
Let $\fN$ be a non-empty totally ordered subset of $\fM(\Lambda;x_0,M_0)$.
Let $N\in \fN$.
Since $e$ obviously belongs to $\cE_\succeq$, we conclude that $\Lambda(e)\geq \tr(e(x_0)N)\geq \varepsilon\, \tr N$, so that $\tr N\leq \varepsilon^{-1}\Lambda (e)$.
This implies that the set $\fN$ is bounded from above, so Corollary~\ref{totallycor} applies.
Hence there exist $N_0\in \Hqgg $ and a sequence $(N_n)_{n\in \dN}$ from $\fN$ such that $N_0\succeq N$ for all $N\in \fN$ and $N_0=\lim_n N_n$.
Then, for each $n\in \dN$, we have $\Lambda (F)\geq \tr(F(x_0)N_n)$ for all $F\in \cE_\succeq$ and $N_n\succeq M_0$.
Passing to the limit we obtain $\Lambda(F)\geq \tr(F(x_0)N_0)$ for all $F\in \cE_\succeq$ and $N_0\succeq M_0$.
Thus $N_0\in \fM(\Lambda;x_0,M_0)$.
Therefore, by Zorn's lemma, $\fM(\Lambda;x_0,M_0)$ has at least one maximal element.
Obviously, $O\in \fM(\Lambda;x_0,O)$.
In particular, $\fM(\Lambda;x_0,O)\neq\emptyset$, so that $\fM(\Lambda;x_0,O)$ as at least one maximal element.

Now we verify the second assertion.
Assume that $M_1\neq O$ is a maximal element of $\fM(\Lambda;x_0,O)$.
Then, $M_1\in\Hqgg$.
First, for $M\in \Hqgg $ we define 
\begin{equation}\label{defim}
    I(M)
    :=\inf_{F\in \cE_{\succeq}}\frac{\Lambda(F)}{\tr(F(x_0)M)},
\end{equation}
where we set $\frac{a}{0}:=+\infty$ for any $a\in [0,+\infty)$.
Clearly, $M\in \fM(\Lambda;x_0,O)$ if and only if $I(M)\geq 1$.
Thus, $I(M_1)\geq 1$.
We show that $I(M_1)=1$.
Since $M_1\succeq O$ and $M_1\neq O$ by assumption, we have $\tr(e(x_0)M_1)\neq 0$ and therefore $I(M_1)$ is finite in \eqref{defim}.
Then $\Lambda (F)\geq I(M_1)\tr(F(x_0)M_1)=\tr(F(x_0)[I(M_1)M_1])$ for all $F\in \cE_{\succeq}$ implies that $I(M_1)M_1\in \fM(\Lambda;x_0,O)$.
Regarding $M_1\neq O$, hence $I(M_1)>1$ would contradict the maximality of $M_1$.
Thus we have proved that $I(M_1)=1$. 
 
From the definition of $I(M_1)$ it follows that there exists a sequence $(F_n)_{n\in \dN}$ from $ \cE_{\succeq}$ such that
\[
    \lim_{n\to \infty} \frac{\Lambda(F_n)}{\tr(F_n(x_0)M_1)}
    =I(M_1)
    =1.
\]
We can assume that $\tr(F_n(x_0)M_1)>0$ for all $n\in \dN$.
By scaling $F_n$ we obtain $\tr(F_n(x_0)M_1)=1$ for $n\in \dN$ and $\lim_n \Lambda(F_n)=1$. 
\end{proof}

\begin{rem}\label{R1308}
If $M_0\in\Hqgg$ and $M\in\bM(\Lambda;x_0)$ satisfies $M\succeq M_0$, then we have $M\in\fM(\Lambda;x_0,M_0)$ by Lemma~\ref{L0934}.
In particular, $\bM(\Lambda;x_0)\subseteq\fM(\Lambda;x_0,O)$.
\end{rem}

For the following results in this section we assume that condition~(PC) holds:
\begin{PC}
    A linear functional $\Lambda$ on $\cE$ is a moment functional if $\Lambda(F)\geq 0$ for all $F\in \cE_\succeq$. 
\end{PC}

Recall that in Theorem~\ref{compactpos} we have shown that condition~(PC) is fulfilled if $\cX$ is a compact space, each $F\in \cE$ is continuous on $\cX$, and there exists $e\in \cE$ such that $e(x)\succeq \varepsilon I_q$ on $\cX$ for some $\varepsilon>0$.

\begin{prop}\label{lpos}
For $Y\in \Hqgg $, we have $Y\in \bP(\Lambda;x_0)$ if and only if 
\begin{equation}\label{cond2}
    \tr(F(x_0)Y)
    \leq \Lambda(F)
    \quad\text{for all }F\in \cE_\succeq.
\end{equation}
That is, $\bP(\Lambda;x_0)=\fM(\Lambda;x_0,O)$.
\end{prop}
\begin{proof} 
First we assume that \eqref{cond2} holds.
From this inequality and condition~(PC) it follows that the linear functional $\tilde{\Lambda}\colon\cE\to \dR$ defined by $\tilde{\Lambda}(F):=\Lambda(F)-\tr(F(x_0)Y)$ is a moment functional.
Let $\tilde\mu$ be a representing measure of $\tilde{\Lambda}$.
Then $\mu:= \tilde\mu + Y\delta_{x_0}$ is a representing measure of $\Lambda$, so $\mu(\{x_0\})\in \bM(\Lambda;x_0)$ and $\mu(\{x_0\})= \tilde\mu(\{x_0\})+ Y\succeq Y$.
Thus, $Y\in \bP(\Lambda;x_0)$.
 
Now we show the converse direction.
Assume that $Y\in \bP(\Lambda;x_0)$.
Then there exists $X\in \bM(\Lambda;x_0)$ such that $X\succeq Y$.
Let $F\in \cE_\succeq$.
From $F(x_0)\succeq O$ and $X-Y\succeq O$ (by $X\succeq Y$), we conclude that $\tr(F(x_0)(X-Y)) \geq 0$, so that $\tr(F(x_0)Y)\leq \tr(F(x_0)X)$.
Combined with Lemma~\ref{L0934}, we obtain
\[
    \tr(F(x_0)Y)
    \leq \tr(F(x_0)X)
    \leq \Lambda(F),
\]
which proves condition \eqref{cond2}.
\end{proof}

A sufficient condition for $M$ being a mass of some representing measure of $\Lambda$ is the following.

\begin{cor}\label{C1420}
Let $M\in \Hqgg$ be such that $\tr(F(x_0)M)\leq \Lambda(F)$ for all $F\in \cE_\succeq$.
Suppose that $M'\succeq M$ and
\begin{equation}\label{cond4}
    \tr(F(x_0)M')
    \leq \Lambda(F)
    \quad\text{for all }F\in \cE_\succeq
\end{equation}
imply $M'=M$.
Then $M\in \bM(\Lambda;x_0)$.
\end{cor}
\begin{proof}
By Proposition~\ref{lpos} we have $M\in\bP(\Lambda;x_0)$, i.\,e., there is a matrix $X\in \bM(\Lambda;x_0)$ such that $X\succeq M$.
But $M':=X$ satisfies \eqref{cond4} as well, so $X=M'=M$ by the assumption.
\end{proof}

\begin{lem}\label{L1419}
Let $x_0\in \cX$.
Then\, $\bM_\mathrm{max}(\Lambda;x_0)=\fM_\mathrm{max}(\Lambda;x_0,O)\neq\emptyset$.
\end{lem}
\begin{proof}
Lemma~\ref{auxl} yields $\fM_\mathrm{max}(\Lambda;x_0,O)\neq\emptyset$.

First let $M_1\in\fM_\mathrm{max}(\Lambda;x_0,O)$.
Let $M$ be an element of $\bM(\Lambda;x_0)$ such that $M\succeq M_1$.
Remark~\ref{R1308} implies $M\in\fM(\Lambda;x_0,O)$.
Hence, $M=M_1$ by the maximality of $M_1$.
Since $M_1\in\fM(\Lambda;x_0,O)$, Proposition~\ref{lpos} yields $M_1\in \bP(\Lambda;x_0)$, so there exists $M'\in \bM(\Lambda;x_0)$ such that $M'\succeq M_1$.
As above, we can infer then $M'=M_1$. In particular, this shows that $M_1\in\bM(\Lambda;x_0)$.
Consequently, $M_1\in \bM_\mathrm{max}(\Lambda;x_0)$.

Now let $M_1\in\bM_\mathrm{max}(\Lambda;x_0)$.
Then, $M_1\in\bM(\Lambda;x_0)$, so that Remark~\ref{R1308} implies $M_1\in\fM(\Lambda;x_0,O)$.
Take an element  $M\in \fM(\Lambda;x_0,O)$ such that $M\succeq M_1$.
Then Proposition~\ref{lpos} shows $M\in \bP(\Lambda;x_0)$, so there exists $M'\in\bM(\Lambda;x_0)$ such that $M'\succeq M$.
Hence $M'\succeq M\succeq M_1$. This implies $M'=M_1$ by the maximality of $M_1$.
In particular, $M=M_1$ follows.
Consequently, $M_1\in\fM_\mathrm{max}(\Lambda;x_0,O)$.
\end{proof}

\begin{lem}\label{L1346}
Let $x_0\in \cX$ and let $M_0\in \bM(\Lambda;x_0)$.
Then
\[
    \emptyset
    \neq\fM_\mathrm{max}(\Lambda;x_0,M_0)
    =\{M_1\in\bM_\mathrm{max}(\Lambda;x_0)\colon M_1\succeq M_0\}.
\]
\end{lem}
\begin{proof}
Remark~\ref{R1308} yields $M_0\in\fM(\Lambda;x_0,M_0)$.
Hence, $\fM(\Lambda;x_0,M_0)\neq\emptyset$.
Because of Lemma~\ref{auxl}, then $\fM_\mathrm{max}(\Lambda;x_0,M_0)\neq\emptyset$.

First we let $M_1\in\fM_\mathrm{max}(\Lambda;x_0,M_0)$.
Then $M_1\in\fM(\Lambda;x_0,M_0)$, so that $M_1\succeq M_0$ by Definition~\ref{D0709}.
Let $M$ be an  element of $ \bM(\Lambda;x_0)$ such that $M\succeq M_1$.
Then $M\succeq M_0$, so that Remark~\ref{R1308} implies $M\in\fM(\Lambda;x_0,M_0)$.
Therefore, $M=M_1$ by the maximality of $M_1$.
Since $M_1\in\fM(\Lambda;x_0,M_0)$, we have $M_1\in\fM(\Lambda;x_0,O)$ by \eqref{G0814}.
Hence $M_1\in \bP(\Lambda;x_0)$ by Proposition~\ref{lpos}, so there exists $M'\in \bM(\Lambda;x_0)$ such that $M'\succeq M_1$.
As above, we conclude that $M'=M_1$. In particular, $M_1\in\bM(\Lambda;x_0)$, so that $M_1\in \bM_\mathrm{max}(\Lambda;x_0)$.

Now let $M_1\in\bM_\mathrm{max}(\Lambda;x_0)$ be such that $M_1\succeq M_0$.
Then, $M_1\in\bM(\Lambda;x_0)$, so that Remark~\ref{R1308} implies $M_1\in\fM(\Lambda;x_0,M_0)$.
Let $M$ be an arbitrary element of $\fM(\Lambda;x_0,M_0)$ such that $M\succeq M_1$.
According to \eqref{G0814}, in particular $M\in\fM(\Lambda;x_0,O)$.
Then Proposition~\ref{lpos} shows $M\in \bP(\Lambda;x_0)$, so there exists $M'\in\bM(\Lambda;x_0)$ such that $M'\succeq M$.
Hence $M'\succeq M\succeq M_1$ which implies that $M'=M_1$ by the maximality of $M_1$.
In particular, $M=M_1$ follows.
Consequently, $M_1\in\fM_\mathrm{max}(\Lambda;x_0,O)$.
\end{proof}

\begin{prop}\label{maxdominating}
For each $x_0\in \cX$ and $M_0\in \bM(\Lambda;x_0)$ there exists a matrix $M_1\in \bM_\mathrm{max}(\Lambda;x_0)$ such that $M_1\succeq M_0$.
\end{prop}
\begin{proof}
This is an immediate consequence of Lemma~\ref{L1346}.
\end{proof}

It seems that not much is known about the set $\bM(\Lambda;x_0)$ of masses.
The following natural and important questions are worth to be studied:

\begin{question}\label{Q1436}
When is the set $\bM(\Lambda;x_0)$ closed?
\end{question}

\begin{question}\label{Q1437}
When is $\bM(\Lambda;x_0)$ matrix convex, that is, when does $X_1\succeq X\succeq X_2$ with $X_1,X_2\in \bM(\Lambda;x_0)$ imply that $X\in \bM(\Lambda;x_0)$? 
\end{question}

\section{Ordered maximal mass measures}\label{orderded}
For the scalar truncated moment problem in one variable it is well-known that there exist finitely atomic representing measures such that the mass at each atom is maximal.
In the multi-dimensional scalar case it is a difficult task to find such representing measures even if the space $\cX$ is compact.
To remedy this the weaker notion of ordered maximal mass measures was proposed and an existence result was proved in \cite[Section~18.5]{sch16}.
In this section we essentially use Theorem~\ref{compactpos} to derive a similar result in the matricial case.
The corresponding definition is the following.

\begin{dfn}\label{maxmassmeasure}
Suppose $\Lambda$ is a moment functional on $\cE$ and $\nu=\sum _{j=1}^n  M_j\delta_{x_j}$, $n\in \dN$, is a representing measure for $\Lambda$.
Let $\nu_{k}:=\sum_{j=k}^n M_j\delta_{x_j}$ and $\Lambda_{k}:=\Lambda^{\nu_k}$ for $k=1,\dotsc,n$.
Then $\nu$ is called \emph{ordered maximal mass} for $\Lambda$ if for each $k=1,\dotsc,n$ the following is true:
 
If $\mu_k$ is an arbitrary representing measure for $\Lambda_{k}$ such that $\mu_k(\{x_k\})\succeq M_k$, then $\mu_k(\{x_k\})= M_k$.
\end{dfn} 
 
Recall that the matrix $\mu_k(\{x_k\})$ is called the mass of $\mu_k$ at $x_k$.
Therefore, the latter condition in Definition~\ref{maxmassmeasure} means that $M_k$ is a \emph{maximal} mass at the point $x_k$ for the representing measures of $\Lambda_k$.
In contrast to the scalar case this does not mean that $M_k$ is the \emph{largest} possible mass among all representing measures for $\Lambda_k$.
We will discuss this difference in Section~\ref{largestmasses}.

The set of atoms of $\mu\in\cM_q(\cX,\fX)$ is
\[
    \ats(\mu)
    :=\{x\in\cX\colon\mu(\{x\})\neq O\}
\]
and the set of atoms of a moment functional $\Lambda$ (cf. \cite[Definition~7.1]{ms}) is 
\[
    \cW(\Lambda)
    := \bigcup_{\mu\in\cM_\Lambda}\ats (\mu).
\]
In \cite[Theorem~8.16]{ms} it was shown that the core set $\cV(\Lambda)$ of a moment functional $\Lambda\neq0$ (cf. \cite[Definition~8.2]{ms}) coincides with $\cW(\Lambda)$.

\begin{thm}\label{constrmethod}
Suppose that $\cX$ is a compact topological space, all $F\in \cE$ are continuous on $\cX$, and there are a matrix $e\in \cE$ and an $\varepsilon >0$ such that $e(x)\succeq \varepsilon I_q$ for all $x\in \cX$.
Let $\Lambda \neq 0$ be a moment functional on $\cE$.
Then $\Lambda$ has a representing measure $\nu=\sum_{j=1}^n M_j\delta_{x_j}$, $n\leq \dim \cE$, which is ordered maximal mass.
Further, any point from the set $\cW(\Lambda)=\cV(\Lambda)$ can be taken for $x_1$.
\end{thm}
\begin{proof}
Since $\Lambda\neq 0$, we have $\cW(\Lambda)\neq \emptyset$.
Take a point $x_1\in \cW(\Lambda)$.
Let $M_1$ be a maximal element of the set $\fM(\Lambda;x_1,O)$ from Lemma~\ref{auxl}.
Since $x_1\in \cW(\Lambda)$, $\Lambda_1:=\Lambda $ has a representing measure which has a mass $N_1\neq O$ at $x_1$, so $N_1\in \fM(\Lambda;x_1,O)$ and $N_1\neq O$.
Hence $M_1\neq O$, because $M_1$ is maximal in $\fM(\Lambda;x_1,O)$.
Define $\Lambda_2(F):=\Lambda_1 (F)- \tr(F(x_1)M_1)$ for $F\in \cE$.
If $(F_{1n})_{n\in \dN}$ denotes the corresponding sequence of $\cE_\succeq$ from Lemma~\ref{auxl}, then by \eqref{auxsequ} we have
\begin{align*}
    \lim_{n\to \infty} \Lambda_1(F_{1n})&=1&
    &\text{and}&
    \lim_{n\to \infty} \Lambda_2(F_{1n})&= \lim_{n\to \infty} \Lambda_1(F_{1n})- 1=0.
\end{align*}
Since $\Lambda(F)\geq \tr(F(x_1)M_1)$ for all $F\in \cE_{\succeq}$ by  Definition~\ref{D0709}, we have $\Lambda_2(F)\geq 0$ for all $F\in \cE_{\succeq}$.
Therefore, by Theorem~\ref{compactpos}, $\Lambda_2$ is a moment functional. 

Let $\mu_1$ be an arbitrary representing measure of the moment functional $\Lambda_1= \Lambda$ such that $M:=\mu_1(\{x_1\})\succeq M_1$.
Define $\Lambda'(F):=\Lambda(F)- \tr(F(x_1)M)$, $F\in \cE$.
Then $\Lambda'$ has the representing measure $\mu_1-M\delta_{x_1}$.
In particular, $\Lambda'$ is a moment functional.
Therefore, for all $F\in \cE_{\succeq}$ we have $\Lambda'(F)\geq 0$ and hence $\Lambda(F)\geq \tr(F(x_1)M)$, so that $M\in \fM(\Lambda;x_1,O)$.
Since $M\succeq M_1$ and $M_1$ is maximal in $\fM(\Lambda;x_1,O)$, it follows that $\mu_1(\{x_1\})=M=M_1$.

If $\Lambda_2= 0$, we set $n:=1$, $\nu:=M_1\delta_{x_1}$ and we are done.
If $\Lambda_2\neq 0$, we apply the preceding construction with $\Lambda_1=\Lambda$ replaced by $\Lambda_2$.
Continuing this procedure, we obtain $\kappa\in\dN\cup\{\infty\}$, a sequence $(x_k)_{k=1}^{\kappa}$ of points from $\cX$, a sequence $(M_k)_{k=1}^{\kappa}$ of matrices from $\Hqgg $, all different from $O$, and a sequence $(\Lambda_k)_{k=1}^{\kappa+1}$ of moment functionals on $\cE$, such that the following four conditions are true:
\begin{enumerate}[(I)]
    \item\label{constrmethod.I} $\Lambda_{k+1}(F)=\Lambda_k(F)-\tr(F(x_k)M_k)$ for all $F\in \cE$ and each $k\in\dN$ with $k\leq\kappa$.
    \item\label{constrmethod.II} For each $k\in\dN$ with $k\leq\kappa$, there is a sequence $(F_{kn})_{n\in \dN}$ from $\cE_\succeq$ such that $\lim_{n\to \infty} \Lambda_{k}(F_{kn})=1$ and $\lim_{n\to \infty}\Lambda_{k+1}(F_{kn})=0$.
    \item\label{constrmethod.III} If $k\in\dN$ with $k\leq\kappa$ and $\mu_{k}$ is a representing measure for $\Lambda_k$ such that $\mu_k(\{x_k\})\succeq M_k$, then $\mu_k(\{x_k\})=M_k$.
    \item\label{constrmethod.IV} $\Lambda_k\neq0$ for all $k\in\dN$ with $k\leq\kappa$ and $\Lambda_{\kappa+1}=0$ if $\kappa<\infty$.
\end{enumerate}
Since $\Lambda_s(F)\leq \Lambda_r(F)$ for all $F\in \cE_\succeq$ and $r,s\in\dN$ with $r\leq s\leq\kappa+1$ by~\ref{constrmethod.I}, the second identity in~\ref{constrmethod.II} implies that
\begin{equation}\label{limjk2}
    \lim_{n\to \infty} \Lambda_\ell(F_{kn})
    =0
    \quad\text{for all }k,\ell\in \dN\text{ with }k+1\leq\ell\leq\kappa+1.
\end{equation}

Setting $m:=\dim \cE$, we have $m<\infty$.
Next we show that $\kappa\leq m$.
Assume to the contrary that $\kappa>m$.
According to~\ref{constrmethod.IV} and~\ref{constrmethod.II}, for each $k=1,\dotsc, m+1$, then $\Lambda_k\neq 0$ and the sequence $(F_{kn})_{n\in \dN}$ is defined.
Since $m+1> \dim \cE$, the linear functionals $\Lambda_k$, $k=1,\dotsc,m+1$, on $\cE$ are linearly dependent.
Hence there are reals $\lambda_1,\dotsc,\lambda_{m+1}$, not all zero, such that 
\begin{equation}\label{relationF_j}
    \sum_{k=1}^{m+1} \lambda_k \Lambda_k(F)
    =0
    \quad \text{for all }  F\in \cE.
\end{equation} 
Let $r$ be the smallest index for which $\lambda_r\neq 0$. 
Now we set $F=F_{rn}$ in \eqref{relationF_j}.
Passing to the limit $ n\to \infty$ and using with $k=r$ the first identity in~\ref{constrmethod.II} and \eqref{limjk2}, we obtain $\lambda_r=0$, a contraction.
This proves that $\kappa\leq m$.

Setting $n:=\kappa$, we have then $n\leq\dim\cE$ and, in particular, $\kappa=n<\infty$, so that $\Lambda_{n+1}=0$ by~\ref{constrmethod.IV}.
Regarding $\Lambda_1=\Lambda$ and~\ref{constrmethod.I}, we easily see $\Lambda_{k+1}(F)    =\Lambda(F)-\sum_{j=1}^{k} \tr(F(x_j)M_j)$ for all $F\in \cE$ and $k=1,\dotsc,n$.
For $k=n$ it follows that 
\[
    \Lambda(F)
    = \sum_{j=1}^n \tr(F(x_j)M_j)
    \quad\text{for all } F\in \cE,
\]
implying in turn $\Lambda_k(F)= \sum_{j=k}^n \tr(F(x_j)M_j)$ for all $F\in \cE$ and $k=1,\dotsc,n$.
That is, $\nu:=\sum_{j=1}^n M_j\delta_{x_j}$ is a representing measure for $\Lambda$ and $\nu_{k}:=\sum_{j=k}^n M_j\delta_{x_j}$ fulfills $\Lambda_{k}=\Lambda^{\nu_k}$ for $k=1,\dotsc,n$.
Therefore, by~\ref{constrmethod.III}, $\nu$ is ordered maximal mass.
This completes the proof of Theorem~\ref{constrmethod}.
\end{proof}

\section{Largest Masses}\label{largestmasses}
The scalar version of the following result was noted in \cite{sch16}, see also \cite[Theorem~18.33(i)]{sch17}.

\begin{prop}\label{P1505}
Suppose that $\Lambda$ is a moment functional on $\Hq(E)$ which has a representing measure $\nu=M_0\delta_{x_0}+ \sum_{j=1}^k M_i\delta_{x_j}$.
If there exists a function $f\in E$ such that $f(x)\geq 0$ for all $x\in \cX$, 
\begin{equation}\label{assumptionf}
    f(x_0)> 0\text{ and }f(x_1)=\dotsb=f(x_k)=0,
\end{equation}
then $M_0$ is the largest mass of $\Lambda$ at $x_0$, that is, $M_0\succeq M$ for all $M\in \bM(\Lambda; x_0)$.
\end{prop}
\begin{proof}
Let $M$ be an arbitrary matrix belonging to $\bM(\Lambda; x_0)$.
Then there exists a representing measure $\mu$ of $\Lambda$ satisfying $\mu(\{x_0\})=M$.

If $F\in \Hqgg (E)$ is a matrix such that 
\begin{equation}\label{azeros}
    F(x_1)
    =\dotsb
    =F(x_k)
    =O,
\end{equation}
then $\tr(F(x_j)M_j)=0$ for $j=1,\dotsc,k$, so that 
\[
    \tr(F(x_0)M_0)
    =\Lambda^\nu(F)
    =\Lambda(F)
    \geq \tr( F(x_0)M)
\]
by Lemma~\ref{L0934} and hence
\begin{equation}\label{auxpos}
    \tr(F(x_0)(M_0-M))
    \geq 0.
\end{equation}

Now let $\lambda$ be an eigenvalue of the Hermitian matrix $M_0-M $ with eigenvector $y$. 
We define a matrix $F(x)$ by $F(x):=f(x)yy^*$, where $f$ is as assumed in the proposition.
Since $f(x)\geq 0$ on $\cX$, we have $F\in\Hqgg (E)$.
Because $f$ satisfies \eqref{assumptionf}, it follows that \eqref{azeros} and hence \eqref{auxpos} hold.
We derive
\[\begin{split}
    \tr(F(x_0)(M_0-M ))
    &=\tr(f(x_0)yy^*(M_0-M ))
    =\tr(f(x_0)y^*(M_0-M )y)\\
    &=\tr(f(x_0) \lambda y^*y)
    =f(x_0)\lambda\lVert y\rVert^2
    \geq 0.
\end{split}\]
Since $y\neq 0$ and $f(x_0)>0$, we conclude that $\lambda \geq 0$.
Thus all eigenvalues of $M_0-M $ are non-negative, so that $M_0-M \succeq O$.
Therefore, $M_0\succeq M $.
\end{proof} 

In the remaining part of this section we construct moment functionals which have no largest mass at some point.
We begin with a finitely atomic \emph{scalar} measure 
\[
     \nu
     =\sum_{j=0}^{k} m_j \delta_{x_j}, 
\]
with $m_j\geq0$ for $j=0,\dotsc,k$ and pairwise different points $ x_j\in \cX$.

We define a linear subspace $\cE_0$ of $\cH_2(E)$ by
\begin{equation}\label{defce}
    \cE_0
    :=
    \left\{F=
     \begin{pmatrix}
         f_{11} & f_{12}\\
         f_{21} & f_{22}
     \end{pmatrix}\in \cH_2(E)\colon f_{12}(x_0)+f_{21}(x_0)=0 \right\}.
\end{equation}
Further, we consider positive semi-definite real matrices 
\[
    M_j
    \equiv
     \begin{pmatrix}
         m_{j} & w_j \\
         w_j & m_{j}
     \end{pmatrix},
     \quad j=0,\dotsc,k,
\]
and define a matrix measure
\[
    \mu
    :=\sum_{j=0}^k M_j\delta_{x_j}
    \in\cM_2(\cX,\fX).
\] 

\begin{prop}\label{nolargest}
Suppose that $m_0>0$,  $M_0$ is positive definite, and  the moment functional $l^\nu$ of $\nu$ on $E$ has maximal mass $m_0$ at $x_0$.
Then the moment functional $\Lambda^\mu$ on $\cE_0$ has no largest mass matrix at $x_{0}$. 
\end{prop}
\begin{proof}
Since $M_0$ is positive definite, $ m_{0}^2-w_0^2>0$.
Hence there exists an $\varepsilon >0$ such that for all $\alpha\in \dR$ satisfying $\lvert \alpha\rvert \leq \varepsilon$ we have
\[
    m_{0}^2-(w_0+\alpha)^2
    >0.
\]
Then the matrix
\[
    M_0(\alpha)
    :=
    \begin{pmatrix}
        m_{0} & w_0+\alpha\\
        w_0+\alpha& m_{0}
    \end{pmatrix}
\]
is also positive definite for all $\lvert \alpha\rvert \leq \varepsilon$.
Define
\[
    \mu_\alpha
    :=M_0(\alpha)\delta_{x_0}+\sum_{j=1}^k M_j\delta_{x_j}
    \in \cM_2(\cX,\fX).
\] 

Let $F\in \cE_0$.
Recall that $f_{12}(x_0)+f_{21}(x_0)=0$ by \eqref{defce}.
Using this equality we derive
\[\begin{split}
    \Lambda^{\mu_\alpha}(F)-\Lambda^\mu(F)
    &=\tr (F(x_0) M_0(\alpha))-\tr (F(x_0) M_0)\\
    &=\tr (F(x_0) [M_0(\alpha)-M_0])
    =f_{21}(x_0)+f_{12}(x_0)
    =0,
\end{split}\]
so $\mu_\alpha$ is a representing measure for the moment functional $\Lambda^\mu$ on $\cE_0$.

Now assume to the contrary that $\Lambda^\mu$ has a largest mass matrix at $x_0$, say 
\[
    A
    :=
    \begin{pmatrix}
         a & b\\
         \overline{b} & c \\
    \end{pmatrix},
\]
which is obtained for some measure $\mu'\in \cM_2(\cX,\fX)$.
Regarding \cite[Lemma~7.2]{ms}, we can assume that $\mu'$ is of the form $\mu'=A\delta_{x_0}+\sum_{i=1}^m N_i\delta_{y_i}$, where $y_i\neq x_0$ for $i=1,\dotsc,m$.
Since $\Lambda^\mu=\Lambda^{\mu'}$, the corresponding scalar moment functionals in the left upper corner coincide.
This implies that 
\begin{equation}\label{eauitau}
     l^\nu(f)
    =m_0f(x_0)+\sum_{j=1}^k m_{j}f(x_j)
    =af(x_0)+\sum_{i=1}^m n_{i}f(y_i)\quad \text{for all }f\in E,
\end{equation} 
where $n_{i}$ denotes the left upper corner entry of $N_i$.
By assumption, $m_0$ is the maximal mass of the moment functional $ l^\nu$ on $E$ at $x_0$.
Therefore, by \eqref{eauitau}, $m_0\geq a$.
On the other hand, by assumption $A$ is the largest mass matrix of $\Lambda^\mu$ at $x_0$.
Therefore, $A\succeq M_0$ and hence $a\geq m_0$.
Thus, $m_0=a$. 

Since $\Lambda^{\mu_\alpha}=\Lambda^\mu$ as shown above and $A$ is the largest mass matrix at $x_0$, we have
\[
    A-M_0(\alpha)
    =
    \begin{pmatrix}
         0 & b- (w_{0}+\alpha)\\
         \overline{b}-(w_{0}+\alpha)& 0 \\
    \end{pmatrix}
    \succeq O.
\]
Therefore, $b- (w_{0}+\alpha)=0$ for all reals $\alpha$ such that $\lvert \alpha\rvert \leq \varepsilon$, a contradiction.
This completes the proof.
\end{proof}

The following very simple examples show how extreme cases for  sets of masses  occur.

\begin{exm}\label{E1012}
Let $\cX:=\{0\}$ and let $E$ be the linear span of  the polynomial $1$.
Then $k=0$, $x_0=0$, and  
\[
    \cE_0
    =\left\{\begin{pmatrix}\alpha &-\im\beta \\\im\beta&\gamma\end{pmatrix}\colon \alpha,\beta,\gamma\in\dR\right\}.
\]
Further, we set $m_0=1$, i.\,e., $\nu=\delta_0$, and  $w_0=0$, i.\,e., $M_0=I_2$, so that $\mu=I_2\delta_0$.
Then the assumptions of Proposition~\ref{nolargest} are fulfilled.

In this case, the set of  representing measures of the moment functional $\Lambda:=\Lambda^\mu$ on $\cE_0$  consists of all  measures $\mu'=A_\eta\delta_0$, where $\eta\in \dR$, $\lvert\eta\rvert\leq 1$, and
\begin{equation*}\label{E1012.2}
    A_\eta
    :=
    \begin{pmatrix}1 &\eta \\\eta&1\end{pmatrix}
    =M_0+\eta\begin{pmatrix}0 &1\\1&0\end{pmatrix}.
\end{equation*} 
Clearly, any such $\mu'=A_\eta\delta_0$ is a representing measure of $\Lambda$.
Let $\mu'$ be an arbitrary representing measure of $\Lambda$.
Since $\cX=\{0\}$, we have $\mu'=A\delta_0$ for some matrix $A=(\begin{smallmatrix}a &b \\\overline{b}&c\end{smallmatrix})\in \cH_{2,\succeq}$.
The equality $\Lambda=\Lambda^{\mu'}$ means  that $\alpha+\gamma=\alpha a +\im \beta\, b -\im \beta\, \overline{b} +\gamma c$ for all $\alpha,\beta,\gamma\in\dR$.
The latter is equivalent to $a=c=1$ and $b\in \dR$, i.\,e., $A=A_b$.
Since $A\succeq O$, $\lvert b\rvert\leq 1$.

Thus, $\bM(\Lambda;x_0)$ is the set of matrices $A_\eta$, where $\eta\in[-1,1]$.
Hence the set $\bM(\Lambda;x_0)$ is closed and it is an antichain with respect to ``$\succeq$'' (that is, if $X_1,X_2\in \bM(\Lambda;x_0)$ with $X_1\neq X_2$, then $X_1\nsucceq X_2$ and $X_2\nsucceq X_1$). In particular, we have $\bM_\mathrm{max}(\Lambda;x_0)=\bM(\Lambda;x_0)$.
\end{exm} 

\begin{exm}\label{E1802}
Suppose $\cX$ consists of the following four points of $\dR^2$: 
\begin{align*}
    x_0&=(0,0),&
    x_1&=(1,0),&
    x_2&=(0,1),&
    x_3&=(1,1).
\end{align*}
Let $E$ be the linear span of  polynomials $1,x,x^2,y,y^2$.
We easily compute 
\begin{equation}\label{linrela}
    \ell_{x_0}+\ell_{x_3}
    =\ell_{x_1}+\ell_{x_2}.
\end{equation}
Suppose that $M_0, M_1, M_2, M_3$ are matrices of $\cH_{2,\succeq}$ and set
\[
    \mu
    :=M_0\delta_{x_0}+	M_1 \delta_{x_1}+	M_2 \delta_{x_2}+	M_3\delta_{x_3}.
\]
Let $\mu'$ be an arbitrary representing measure of the moment functional $L:=L^\mu$ on $E$.
Since $\cX=\{ x_0,x_1,x_2,x_3\}$, then we can assume that 
\begin{equation}\label{muprime}
    \mu'
    =M_0'\delta_{x_0}+	M_1' \delta_{x_1}+	M_2' \delta_{x_2}+	M_3'\delta_{x_3}.
\end{equation} 
Let $M_k=(m_{ij;k})_{i,j=1}^2$ and $M_k'=(m_{ij;k}')_{i,j=1}^2$ for $k=0,1,2,3$.
The equality $L^{\mu'}=L$ means that for $i,j=1,2$ we have
\[
    \sum_{k=0}^3 m_{ij;k}' f(x_k)
    =\sum_{k=0}^3 m_{ij;k} f(x_k)
    \quad\text{for all } f\in E.
\]
A simple computation shows that this equation implies that there is a real number $a_{ij}$ such that $m_{ij;k}'= m_{ij;k}+a_{ij}$ for $k=0,3$ and $m_{ij;k}'= m_{ij;k}-a_{ij}$ for $k=1,2$.
Therefore, if $A$ denotes the matrix $(a_{ij})_{i,j=0}^2$, then 
\begin{align}\label{mprime}
    M_0'&=M_0+A,&
    M_1'&=M_1-A,&
    M_2'&=M_2-A,&
    M_3'&=M_3+A.
\end{align}
In particular, $A\in \cH_2$.
Conversely, from \eqref{linrela} it follows that for any matrix $A\in \cH_2$ such that the matrices $M_0', M_1', M_2', M_3'$ defined by \eqref{mprime} are in $\cH_{2,\succeq}$, equation \eqref{muprime} gives a representing measure $\mu'$ of $L$.
Consequently, all representing matrix measures $\mu'$ of $L$ are known.
They are parametrized by the matrices $A\in \cH_2$ for which all four matrices $M_k'$ given by \eqref{mprime} are in $\cH_{2,\succeq}$, or equivalently,
\begin{equation}\label{inequmi}
    M_1
    \succeq A\succeq -M_0
    \text{ and }
    M_2
    \succeq A\succeq - M_3.
\end{equation}

Regarding Definition~\ref{D2116} and \cite[Proposition~4.5]{ms}, thus $\bM(\Lambda_L;x_0)$ is the set of matrices $M_0+A$, where $A\in \cH_2$ satisfies \eqref{inequmi}.
In particular, this implies that the set $\bM(\Lambda_L;x_0)$ is closed and order-convex (that is, if $X_1\preceq X\preceq X_2$ with $X_1,X_2\in \bM(\Lambda_L;x_0)$, then $X\in \bM(\Lambda_L;x_0)$).

For instance, in the special case $M_0=M_1=M_2=M_3:=I$, we have
\[
    \bM(\Lambda_L;x_0)
    =\{M\in \cH_2\colon: O\preceq M\preceq 2I \}
    =[O,2 I],
\]
so that $\bM_\mathrm{max}(\Lambda_L;x_0)=\{2I\}$.
\end{exm}

\section{Commutative matrix moment functionals}\label{commuative}
In this section we study the very special class of matrix moment functionals $L$ on $E$ for which $L(f)L(g)=L(g)L(f)$ for all $f,g\in E$.

We begin with a number of standard notions.

Let $L_1\colon E\to \cH_{q_1}, \dotsc, L_m\colon E\to \cH_{q_m}$, $m\in \dN$, be linear mappings.
The linear mapping $L_1\oplus\dotsb \oplus L_m\colon E\to \cH_{q_1+\dotsb+q_m}$ defined by
\[
    (L_1\oplus\dotsb \oplus L_m)(f)
    :=L_1(f)\oplus \dotsb\oplus L_{m}(f),
    \quad f\in E,
\]
is called the \emph{direct sum} of the linear mappings $L_1,\dotsc,L_m$. 
 
We say two mappings $L,L'\colon E\to \Hq$ are \emph{unitarily equivalent} if there exists a unitary matrix $U\in M_q(\dC)$ such that $L'(f)=U^*L(f)U$ for all $f\in E$.
 
Matrix moment functionals $L\colon E\to \cH_1\cong \dR$ are called \emph{scalar} moment functionals.
We will denote scalar moment functionals by $l$.
In particular, $ l^\mu$ is the moment functional on $E$ with representing measure $\mu$.

\begin{dfn}\label{defcodi}
A matrix moment functional $L$ on $E$ is called
\begin{itemize}
    \item \emph{commutative} if $L(f)L(g)=L(g)L(f)$ for all $f,g\in E$
    \item \emph{diagonalizable} if it is unitarily equivalent to a direct sum of scalar moment functionals.
\end{itemize}
\end{dfn}

\begin{prop}\label{L0836}
For any moment functional $\Lambda\neq 0$ on $\cE$ the following statements are equivalent:
\begin{enumerate}[(i)]
    \item\label{L0836.ii} There exists an orthonormal basis $\{u_1,\dotsc,u_q\}$ of $\dC^q$ and scalar moment functionals $l_1,\dotsc,l_q$ on $E:=\{v^*Fv\colon F\in\cE,\,v\in\dC^q\}$, such that $\Lambda(F)=\sum_{r=1}^ql_r(u_r^*Fu)$ for all $F\in\cE$. 
    \item\label{L0836.iii} There exists a finitely atomic representing measure $\nu=\sum_{j=1}^k M_j\delta_{x_j}$ of $\Lambda$ such that $M_iM_j=M_jM_i$ for all $i,j=1,\dotsc,k$.
    \item\label{L0836.iv} There exists a representing measure $\mu$ of $\Lambda$, with Radon--Nikodym matrix $\Phi$ with respect to the trace measure, such that $\Phi(x)\Phi(y)=\Phi(y)\Phi(x)$ for all $x,y\in \cX$.
\end{enumerate}
\end{prop}
\begin{proof}
\ref{L0836.ii}$\to$\ref{L0836.iii}:
By the (scalar) Richter--Tchakaloff theorem, each moment functional $ l_r$ has a finitely atomic representing measure $\nu_r$.
Let $\{x_1,\dotsc,x_k\}$ be the set of all atoms of these $q$ measures.
Then each measure $\nu_r$ is of the form $\nu_r=\sum_{j=1}^k c_{rj}\delta_{x_j}$, where $c_{rj}\geq 0$ for $r=1,\dotsc,q$, $j=1,\dotsc, k$.
(Note that if $x_j$ is not an atom of $\nu_r$, we set $c_{rj}:=0$.)
Let $C_j\in \Hqgg $ denote the diagonal matrix with diagonal entries $c_{1j},\dotsc,c_{qj}$ and put $M_j:=UC_jU^*=\sum_{r=1}^qc_{rj}u_ru_r^*$, where $U:=(u_1,\dotsc,u_q)$.
Since the diagonal matrices $C_j$ pairwise commute and $U$ is unitary, the matrices $M_j$ pairwise commute as well.
Furthermore, $\nu:=\sum_{j=1}^k M_j\delta_{x_j}$ is a representing measure of $\Lambda$, since
\[\begin{split}
    \Lambda^\nu(F)
    &=\sum_{j=1}^k\tr(F(x_j)M_j)
    =\sum_{j=1}^k\sum_{r=1}^qc_{rj}\tr(F(x_j)u_ru_r^*)\\ &=\sum_{r=1}^q\sum_{j=1}^kc_{rj}u_r^*F(x_j)u_r
    =\sum_{r=1}^ql_r(u_r^*Fu_r)
    =\Lambda(F)
\end{split}\]
for all $F\in\cE$ by~\ref{L0836.ii}.
That is,~\ref{L0836.iii} is proved. 
 
\ref{L0836.iii}$\to$\ref{L0836.ii}: 
The matrices $M_1,\dotsc,M_k$ belong to $\Hqgg$ and  pairwise commute by~\ref{L0836.iii}.
Therefore, these matrices have a common set of eigenvectors, say $u_1,\dotsc,u_q$, which form an orthonormal basis of $\dC^q$.
Then, for $j=1,\dotsc,k$, there exist non-negative real numbers $ m_{j1},\dotsc, m_{jq}$ such that $M_j=\sum_{r=1}^q m_{jr}u_ru_r^*$. 
For $r=1,\dotsc,q$, we see that $l_r\colon E\to\dR$ defined by $l_r(f):=\sum_{j=1}^k m_{jr}f(x_j)$ is a scalar moment functional, since $ m_{jr}\geq0$.
Using~\ref{L0836.iii} we derive 
\[\begin{split}
    \Lambda(F)
    =\Lambda^\nu(F)
    &=\sum_{j=1}^k\tr(F(x_j)M_j)
    =\sum_{j=1}^k\sum_{r=1}^q m_{jr}\tr(F(x_j)u_ru_r^*)\\
    &=\sum_{r=1}^q\sum_{j=1}^k m_{jr}u_r^*F(x_j)u_r
    =\sum_{r=1}^ql_r(u_r^*Fu_r)
\end{split}\]
for all $F\in\cE$.
Consequently,~\ref{L0836.ii} is proved. 

\ref{L0836.iii}$\to$\ref{L0836.iv}:
Without loss of generality we can assume that $M_j\neq O$ for all $j$.
Since $M_j\succeq O$, then $\tr M_j\neq 0$ and for the Radon--Nikodym matrix $\Phi$ of $\nu$ with respect to the trace measure of $\nu$ we have $\Phi(x_j)=(\tr M_j)^{-1} M_j$ for $j=1,\dotsc,k$ and $\Phi(x)=O$ otherwise (see \cite[Example~2.2]{ms}).
Thus,~\ref{L0836.iii} implies~\ref{L0836.iv}.

\ref{L0836.iv}$\to$\ref{L0836.iii}: 
In the proof of Theorem~5.1 given in \cite{ms} it was shown that $\Lambda^\mu$ has a finitely atomic representing measure $\nu$ such that all masses are of the form $M_j=m_j\Phi(y_j)$, see formula (5.5) in \cite{ms}.
Hence~\ref{L0836.iv} implies~\ref{L0836.iii}.
\end{proof}

The following theorem shows that a number of natural conditions on matrix moment functionals are equivalent.

\begin{thm}\label{cummcond}
For a matrix moment functional $L$, $L\neq O$, on $E$ the following are equivalent:
\begin{enumerate}[(i)]
    \item\label{cummcond.i} $L$ is commutative, that is, $L(f)L(g)=L(g)L(f)$ for all $f,g\in E$.
    \item\label{cummcond.ii} $L$ is diagonalizable. 
    \item\label{cummcond.v} There are an orthonormal basis $\{u_1,\dotsc,u_q\}$ of $\dC^q$ and scalar moment functionals $l_1,\dotsc,l_q$ on $E$ such that $\Lambda_L(F)=\sum_{r=1}^ql_r(u_r^*Fu)$ for all $F\in\cH_q(E)$.
    \item\label{cummcond.iii} There exists a finitely atomic representing measure $\nu=\sum_{j=1}^k M_j\delta_{x_j}$ of $L$ such that $M_iM_j=M_jM_i$ for all $i,j=1,\dotsc,k$.
    \item\label{cummcond.iv} There exists a representing measure $\mu$ of $L$, with Radon--Nikodym matrix $\Phi$ with respect to the trace measure, such that $\Phi(x)\Phi(y)=\Phi(y)\Phi(x)$ for all $x,y\in \cX$.
\end{enumerate}
\end{thm}
\begin{proof}
\ref{cummcond.i}$\to$\ref{cummcond.ii}: 
Let us fix a basis $\{f_1,\dotsc,f_n\}$ of the vector space $E$.
Then $L(f_1),\dotsc,L(f_n)$ are pairwise commuting (by assumption~\ref{cummcond.i}) Hermitian matrices of $\Hq$.
Therefore, these matrices have a common set of eigenvectors, say $v_1,\dotsc,v_q$, which form an orthonormal basis of $\dC^q$.
Then there exist real numbers $\lambda_{ir}$, where $i=1,\dotsc,n$, $ r=1,\dotsc,q$ such that $L(f_i)v_r=\lambda_{ir}v_r$. 

Since $\{f_1,\dotsc,f_n\}$ is a vector space basis of $E$, each $f\in E$ is of the form $f=\sum_{i=1}^n\alpha_if_i$, with $\alpha_1,\dotsc,\alpha_n\in \dR$ uniquely determined by $f$.
For $r=1,\dotsc,q$, we define $ l_r\colon E\to\dR$ by $l_r(f):=\sum_{i=1}^n\alpha_i\lambda_{ir}$.
Then  $L(f)v_r=\sum_{i=1}^n\alpha_iL(f_i)v_r = l_r(f)v_r$ and hence 
\begin{align}\label{Lr1}
    \langle L(f)v_r,v_s\rangle
    = l_r(f)\delta_{rs}
    \quad \text{for all }f\in E\text{ and }r,s=1,\dotsc,q. 
\end{align}
By assumption, $L$ is a matrix moment functional.
Hence, by the matricial Richter--Tchakaloff theorem \cite[Corollary~5.3]{ms}, $L$ has a finitely atomic representing measure $\nu=\sum_{j=1}^k M_j\delta_{x_j}$, where $x_1,\dotsc,x_k\in \cX$ and $M_1,\dotsc,M_k\in \Hqgg $.
Then, for $f\in E$, 
\begin{equation}\label{Lr2}
     l_r(f)
    =\langle L(f)v_r,v_r \rangle
    =\langle L^\nu(f)v_r,v_r\rangle
    =\sum_{j=1}^k f(x_j)\langle M_jv_r,v_r\rangle
    = l^{\nu_r}(f)
\end{equation} 
where $\nu_r:=\sum_{j=1}^k \langle M_jv_r,v_r\rangle \delta_{x_j}$.
Note that $ \langle M_jv_r,v_r\rangle \geq 0$, because $M_j\in \Hqgg $.
By \eqref{Lr2}, $ l_r$ is a (scalar) moment functional on $E$ with representing measure $\nu_r$.
Let $U\in M_q(\dC)$ denote the unitary matrix defined by $Uv_r=e_r$, where $r=1,\dotsc,q$.
Using \eqref{Lr1} it is straightforward to verify that $L(f)=U^*[l_1(f)\oplus\dotsb\oplus l_q(f)]U$ for all $f\in E$.
Thus $L$ is diagonalizable.

\ref{cummcond.ii}$\to$\ref{cummcond.iii}:
Since $L$ is diagonalizable, there exist a unitary matrix $U\in M_q(\dC)$ and (scalar) moment functionals $ l_1,\dotsc, l_q$ on $E$ such that $L=U^*( l_1\oplus\dotsb\oplus l_q)U$.
By the (scalar) Richter--Tchakaloff theorem, each moment functional $ l_r$ has a finitely atomic representing measure $\nu_r$.
Let $\{x_1,\dotsc,x_k\}$ be the set of all atoms of these $q$ measures.
Then each measure $\nu_r$ is of the form $\nu_r=\sum_{j=1}^k c_{rj}\delta_{x_j}$, where $c_{rj}\geq 0$ for all $r=1,\dotsc,q$, $j=1,\dotsc, k$.
(Note that if $x_j$ is not an atom of $\nu_r$, we set $c_{rj}:=0$.)
Let $M_j'\in \Hqgg $ denote the diagonal matrix with diagonal entries $c_{1j},\dotsc,c_{qj}$ and put $M_j:=U^*M_j'U$.
Then, $\nu:=\sum_{j=1}^k M_j\delta_{x_j}$ is a representing measure of $L=U^*( l_1\oplus\dotsb\oplus l_q)U$.
Since the diagonal matrices $M_j'$ pairwise commute, so do the matrices $M_j=U^*M_j'U$.
That is,~\ref{cummcond.iii} is proved. 
 
\ref{cummcond.iii}$\to$\ref{cummcond.i} is clear, because $L(f)=L^\nu(f)=\sum_{j=1}^k f(x_j)M_j$ for all $f\in E$.

\ref{cummcond.v}$\leftrightarrow$\ref{cummcond.iii}$\leftrightarrow$\ref{cummcond.iv}:
We have $E=\{v^*Fv\colon F\in\cH_q(E),\,v\in\dC^q\}$.
Hence, we can apply Proposition~\ref{L0836} to $\Lambda_L$ and use \cite[Proposition~4.5]{ms}.
\end{proof}

\section{Transport of linear functionals by positive maps}\label{transport}
Representing measures of moment functionals and of matrix moment functionals are positive $\Hq$\nobreakdash-valued measures on $(\cX,\fX)$.
In this section we transform moment functionals and measures under positive linear mappings of matrices.

First we recall a number of well-known notions and facts on positive maps.

\begin{dfn}\label{D1826}
A linear map $\phi\colon M_q(\dC)\to M_p(\dC)$ is called \emph{positive} if $\phi(A)\in\mathcal{H}_{p,\succeq}$ for all $A\in\Hqgg $.
\end{dfn}

If $\phi\colon M_q(\dC)\to M_p(\dC)$ is a positive map, then we have $\phi(A^*)=[\phi(A)]^*$ for all $A\in M_q(\dC)$; in particular, $\phi$ maps $\Hq $ into $\mathcal{H}_p$.

For a linear map $\phi\colon M_q(\dC)\to M_p(\dC)$, the \emph{adjoint} $\phi^\dagger$ is defined by 
\[
    \langle \phi(A),B\rangle
    =\langle A,\phi^\dagger(B)\rangle,
    \quad A\in M_q(\dC),\, B\in M_p(\dC).
\]
Then $\phi$ is positive if and only if $\phi^\dagger$ is positive. 

\begin{exm}\label{E1828}
The linear map $M_q(\dC)\owns A\mapsto\phi(A):=\tr A\in\dC$ is positive and its adjoint $\phi^\dagger\colon \dC\to M_q(\dC)$ is given by $\phi^\dagger(b)= b\cdot I_q\in M_q(\dC)$, $b\in \dC$.
\end{exm}

As usual, let $M_{q,p}(\dC)$ denote the $q\times p$~matrices with complex entries.

\begin{exm}\label{exvj}
Let $k\in\dN$ and $V_1,\dotsc,V_k\in M_{q,p}(\dC)$.
Then 
\begin{equation}\label{choi}
    \phi(A)
    :=\sum_{j=1}^kV_j^* AV_j,
    \quad A\in M_q(\dC),
\end{equation}
defines a positive map $\phi\colon M_q(\dC)\to M_p(\dC)$ with adjoint $\phi^\dagger \colon M_p(\dC)\to M_q(\dC)$ given by $\phi^\dagger (B)=\sum_{j=1}^kV_j BV_j^*$, $B\in M_p(\dC)$.
\end{exm}

\begin{exm}\label{E1148}
Let $k\in\dN$ and let $V\in M_{pq,p}(\dC)$.
We define two mappings $\Pi_p\colon M_q(\dC)\to M_{pq}(\dC)$ and $\phi\colon M_q(\dC)\to M_p(\dC)$ by
\begin{align}\label{E1148.1}
    \Pi_p(A)&:=A\oplus\dotsb\oplus A\text{ ($p$~times)},&
    \phi(A)&:=V^* \Pi_p(A)V,
    \quad A\in M_q(\dC).
\end{align} 
Then $\phi$ is a positive map.
\end{exm}

An important class of positive maps are the \emph{completely positive maps}, see e.\,g.\ \cite[Chapter~3]{bhatia}.
Suppose $\phi\colon M_q(\dC)\to M_p(\dC)$ is a completely positive map.
Then, by a theorem of Choi and Kraus \cite[Theorem~3.1.1]{bhatia}, $\phi$ is of the form \eqref{choi} with $k\leq pq$.
Further, by the Stinespring dilation theorem for matrix maps \cite[Theorem~3.1.2]{bhatia}, $\phi$ can be written as in \eqref{E1148.1}.

\begin{exm}\label{R1606}
Let $U$ be a linear subspace of $\dC^q$ and let $P_U$ denote the orthogonal projection on $U$.
Then the linear map $\pi_U \colon M_q(\dC)\to M_q(\dC)$ defined by $ \pi_U (A):=P_U A P_U$ is positive.
Of course, this is a special case of Example~\ref{exvj}.
Note that for $A\in\Hqgg $ we have $\tr\pi_U (A)=0$ if and only if $U\subseteq\ker A$.
\end{exm}

Next we transport functionals via a positive map.
For a linear map $L\colon E\to \Hq$ we define $L_{\otimes}\colon\cH_q(E)\to \cH_q \otimes \cH_q$ by $L_{\otimes}(F):=\sum_{j,k=1}^q H_{jk} \otimes  L(\langle F,H_{jk}\rangle)$ (cf.\ \cite[Definition~3.1]{ms}).

\begin{prop}\label{R1453}
Suppose that $\phi\colon M_q(\dC)\to M_p(\dC)$ is a positive linear map.
Let $L\colon E\to \Hq$ be a linear map and set $L':=\phi\circ L$.
Then we have
\[
    \Lambda_{L'}
    =\Lambda_L\circ\phi^\dagger
    \text{ and if $p=q$, then }
    (L')_\otimes
    =(\id\otimes\phi)\circ L_\otimes.
\]
\end{prop}
\begin{proof}
We develop with respect to the corresponding orthonormal bases:
\begin{align*}
    \phi(H_{jk})&=\sum_{\ell,m=1}^p \langle \phi(H_{jk}),H_{\ell m}\rangle H_{\ell m},\\
    \phi^\dagger (H_{\ell m})&=\sum_{j,k=1}^q\langle \phi^\dagger (H_{\ell m}),H_{jk}\rangle H_{jk}.
\end{align*}
Using these equations and \eqref{defcl} we derive for $F\in\cH_p(E)$, 
\[\begin{split}
    \Lambda_L (\phi^\dagger (F))
    &=\sum_{j,k=1}^q\langle L(\langle \phi^\dagger (F), H_{jk}\rangle),H_{jk}\rangle
    =\sum_{j,k=1}^q\langle L(\langle F,\phi(H_{jk}\rangle), H_{jk}\rangle\\
    &=\sum_{j,k=1}^q\sum_{\ell,m=1}^p\langle \phi(H_{jk}) , H_{\ell m}\rangle ~ \langle L(\langle F,H_{\ell m}\rangle), H_{jk}\rangle\\
    &=\sum_{j,k=1}^q\sum_{\ell,m=1}^p\langle H_{jk},\phi^\dagger (H_{\ell m})\rangle~ \langle L(\langle F,H_{\ell m}\rangle), H_{jk}\rangle\\
    &=\sum_{j,k=1}^q \sum_{\ell,m=1}^p\langle \phi^\dagger (H_{\ell m}),H_{jk}\rangle~ \langle L(\langle F,H_{\ell m}\rangle ), H_{jk}\rangle\\
    &=\sum_{\ell,m=1}^p \langle L(\langle F,H_{\ell m}\rangle),\phi^\dagger (H_{\ell m})\rangle
    =\sum_{\ell,m=1}^p\langle \phi(L(\langle F,H_{\ell m}\rangle)) ,H_{\ell m}\rangle\\
    &=\sum_{\ell,m=1}^p\langle L'(\langle F,H_{\ell m}\rangle),H_{\ell m}\rangle
    =\Lambda_{L'}(F)
\end{split}\]
and if $p=q$,
\begin{multline*}
    (\id \otimes\phi)(L_\otimes(F))
    =(\id \otimes\phi)\left(\sum_{j,k=1}^q H_{jk}\otimes L(\langle F,H_{jk}\rangle)\right)\\
    =\sum_{j,k=1}^q H_{jk}\otimes\phi(L(\langle F, H_{jk}\rangle ))
    =\sum_{j,k=1}^q H_{jk}\otimes L'(\langle F, H_{jk}\rangle)
    =(L')_\otimes(F).\qedhere
\end{multline*}
\end{proof}

\begin{prop}\label{R1511}
Let $\phi\colon M_q(\dC)\to M_p(\dC)$ be a positive linear map.
\begin{enumerate}[(a)]
    \item\label{R1511.a} If $\mu\in\cM_q(\cX ,\mathfrak{X})$, then $\phi\circ\mu\in\cM_p(\mathcal{X},\mathfrak{X})$. 
    \item\label{R1511.b} If $L$ is matrix moment functional on $E$ and $\mu\in\cM_q(\cX ,\mathfrak{X})$ is a representing measure of $L$, then $\phi\circ L$ is also a matrix moment functional on $E$ and $\phi\circ\mu$ is a representing measure of $\phi\circ L$.
\end{enumerate}
\end{prop}
\begin{proof}
\ref{R1511.a} For any $X\in\mathfrak{X}$, we have $\mu(X)\in\Hqgg $ and hence $\phi(\mu(X))\in\mathcal{H}_{p,\succeq}$ by the positivity of the map $\phi$.
Let $(X_n;n\in \dN)$ be a sequence of pairwise disjoint sets from $\mathfrak{X}$.
Then 
\[
    \mu\left(\bigcup_{n=1}^\infty X_n\right)
    =\sum_{n=1}^\infty \mu(X_n)
    =\lim_{m\to\infty}\sum_{n=1}^m\mu(X_n).
\]
By the continuity and linearity of $\phi$ it follows that 
\[
    \phi\left(\mu\left(\bigcup_{n=1}^\infty X_n\right)\right)
    =\lim_{m\to\infty}\phi\left(\sum_{n=1}^m\mu(X_n)\right)
    =\lim_{m\to\infty}\sum_{n=1}^m\phi(\mu(X_n))
    =\sum_{n=1}^\infty\phi(\mu(X_n)).
\]
Therefore, $\phi\circ\mu\in\cM_p(\mathcal{X},\mathfrak{X})$.

\ref{R1511.b} By~\ref{R1511.a}, $\phi\circ\mu\in\cM_p(\mathcal{X},\mathfrak{X})$.
For measurable elementary functions $g$ on $E$ we obtain $\phi(\int_\cX  g\, \dif\mu)=\int_\cX  g\, \dif(\phi\circ\mu)$ by linearity.
Now let $f\in E$.
We approximate $f$ by elementary functions.
Then, by the continuity of $\phi$, 
\[
    (\phi\circ L)(f)
    =\phi(L(f))
    =\phi\left(\int_\cX  f\, \dif\mu\right)
    =\int_\cX  f\, \dif(\phi\circ\mu).\qedhere
\]
\end{proof}

\section{The matricial apolar scalar product and the matricial moment problem}\label{apolar}
For homogeneous polynomials in several variables there is an interplay between the apolar scalar product and the truncated moment problem (see, for instance, \cite[Chapter~19]{sch17} and \cite{vegter}).
In this final section we try to generalize parts of the beautiful theory to the matricial case.

Let $\cH_{q;d,m}$ denote the real vector space of homogeneous polynomials in $d$ variables of degree $m$ with coefficients from $\Hq$ and let
\[
    \mathsf{N}_{d,m}
    :=\{\alpha\in\dN_0^d\colon\alpha_1+\dotsb+\alpha_d=m\},\quad\text{where }m\in\dN.
\]
Note that the vector space $\cH_{q;d,m}$ has the dimension
\[
    \dim\cH_{q;d,m}
    =\lvert\mathsf{N}_{d,m}\rvert\dim\Hq
    =\binom{m+d-1}{d-1}q^2.
\]

First we define the apolar scalar product $[\cdot,\cdot]$ on $\cH_{q;d,m}$.
We abbreviate
\[
    \binom{m}{\alpha}
    =\frac{m!}{\alpha_1!\dotsm\alpha_d!}
    \quad\text{for }\alpha=(\alpha_1,\dotsc,\alpha_d)\in\mathsf{N}_{d,m}.
\]
For homogeneous polynomials
\begin{equation}\label{G19.2}
    P(x)
    =\sum_{\alpha\in\mathsf{N}_{d,m}}\binom{m}{\alpha} A_\alpha x^\alpha
    \quad\text{and}
    \quad Q(x)
    =\sum_{\alpha\in\mathsf{N}_{d,m}}\binom{m}{\beta} B_\beta x^\beta
\end{equation}
of $\cH_{q;d,m}$ we define
\begin{equation}\label{G19.3}
    [P,Q]
    :=\sum_{\alpha\in\mathsf{N}_{d,m}}\binom{m}{\alpha}\langle A_\alpha,B_\alpha\rangle.
\end{equation}
Using the fact that $\langle\cdot, \cdot\rangle$ is a scalar product on $\Hq$ a straightforward verification shows that $[\cdot,\cdot]$ is a (real-valued) scalar product on the real vector space $\cH_{q;d,m}$.
Thus, $(\cH_{q;d,m},[\cdot,\cdot])$ is a finite-dimensional real Hilbert space.

\begin{dfn}\label{D0825}
$[\cdot,\cdot]$ is called the \emph{apolar scalar product} on $\cH_{q;d,m}$.
\end{dfn}

Note that the coefficient $A_\alpha$ of $P\in\cH_{q;d,m}$ in \eqref{G19.2} can be recovered from the apolar scalar product by
\[
    [P,x^\alpha H]
    =\langle A_\alpha,H\rangle,
    \quad H\in\Hq,\,\alpha\in\mathsf{N}_{d,m},
\]
and
\[
    [P, vv^*x^\alpha]
    =v^*A_\alpha v,
    \quad v\in\dC^q,\,\alpha\in\mathsf{N}_{d,m}.
\]
Denote by $(a,b):=a_1b_1+\dotsb+a_db_d$ the standard scalar product of $\dR^d$.

Let $y=(y_1,\dotsc,y_d)^\trn\in\dR^d$.
We denote by $\pi_{y,m}$ the element of $\cH_{1;d,m}$ defined by
\begin{equation}\label{G19.4}
    \pi_{y,m}(x)
    :=(y,x)^m
    =\left(\sum_{j=1}^dy_jx_j\right)^m
    =\sum_{\alpha\in\mathsf{N}_{d,m}}\binom{m}{\alpha}y^\alpha x^\alpha,
\end{equation}
where the last equality holds by the multinomial theorem. 

Now suppose that $F=\sum_{r=1}^s C_r\pi_{\eta_r,m}\in\cH_{q;d,m}$, where $\eta_r\in\dR^d$ and $C_r\in\Hq$, and let $P\in\cH_{q;d,m}$ be as in \eqref{G19.2}.
From \eqref{G19.4} it follows that
\[
    F(x)
    =\sum_{r=1}^sC_r\left(\sum_{\alpha\in\mathsf{N}_{d,m}}\binom{m}{\alpha}\eta_r^\alpha x^\alpha\right)
    =\sum_{\alpha\in\mathsf{N}_{d,m}}\binom{m}{\alpha}\left(\sum_{r=1}^sC_r\eta_r^\alpha\right)x^\alpha
\]
and hence
\[\begin{split}
    [P,F]
    &=\sum_{\alpha\in\mathsf{N}_{d,m}}\binom{m}{\alpha}\left\langle A_\alpha,\sum_{r=1}^sC_r\eta_r^\alpha\right\rangle
    =\sum_{r=1}^s\left\langle\sum_{\alpha\in\mathsf{N}_{d,m}}\binom{m}{\alpha}A_\alpha\eta_r^\alpha, C_r\right\rangle\\
    &=\sum_{r=1}^s\left\langle P(\eta_r), C_r\right\rangle
    =\sum_{r=1}^s \tr (P(\eta_r)C_r).
\end{split}\]
In particular,
\begin{equation}\label{G19.6}
    [P,B\pi_{y,m}]
    =\langle P(y),B\rangle,
    \quad B\in\Hq,\,y\in\dR^d,\,P\in\cH_{q;d,m}
\end{equation}
and
\[
    [P,vv^*\pi_{y,m}]
    =v^*P(y)v,
    \quad v\in\dC^q,\,y\in\dR^d,\,P\in\cH_{q;d,m}.
\]
In particular, setting $P=A\pi_{a,m}$ and $y=b$ in \eqref{G19.6}, we obtain
\[
    [A\pi_{a,m},B\pi_{b,m}]
    =\langle A(a,b)^m,B\rangle
    =(a,b)^m\langle A,B\rangle,
    \quad a,b\in\dR^d,\,A,B\in\Hq.
\]
Recall $(a,b)$ means the scalar product in $\dR^d$.
Therefore, if $a\perp b$ in $\dR^d$ or $A\perp B$ in $\Hq$, then $A\pi_{a,m}\perp B\pi_{b,m}$ in the Hilbert space $(\cH_{q;d,m},[\cdot,\cdot])$.

Now we develop the interplay between the apolar scalar product and the action of differential operators.

We denote by $\cH_{q;d}[\underline{x}]$ the real vector space of polynomials in $d$ variables with coefficients from $\Hq$.
Let $i\in\{1,\dotsc,d\}$ and $\alpha=(\alpha_1,\dotsc,\alpha_d)\in\dN_0^d$.
Then we abbreviate $\partial_i=\partial/\partial x_i$ and $\partial^\alpha=\partial_1^{\alpha_1}\dotsm\partial_d^{\alpha_d}$.

Given $R(x)=\sum_\gamma x^\gamma R_\gamma\in\cH_{q;d}[\underline{x}]$, we consider $R(\partial)$ as a differential operator acting on $\Hq$\nobreakdash-valued polynomials by
\[
    R(\partial)S
    :=\sum_\gamma \langle R_\gamma,\partial^\gamma S\rangle,
    \quad S\in\cH_{q;d}[\underline{x}].
\]

\begin{lem}\label{L0830}
Let $\eta_1,\dotsc,\eta_s\in\dR^d$ and $C_1,\dotsc,C_s\in\Hq$.
Define
\[
    F
    :=\sum_{r=1}^s C_r\pi_{\eta_r,m}
    \in\cH_{q;d,m}.
\]
If $P\in\cH_{q;d,n}$ and $n\leq m$, then
\begin{equation}\label{G19.14}
    P(\partial)F
    =m(m-1)\dotsm(m+1-n)\sum_{r=1}^s\left\langle P(\eta_r),C_r\right\rangle\pi_{\eta_r,m-n}.
\end{equation}
Let $y\in\dR^d$, $B\in\Hq$, and $P\in\cH_{q;d,m}$.
Then
\begin{equation}\label{G19.15}
    P(\partial)( B\pi_{y,m})
    =m!\left\langle P(y), B\right\rangle.
\end{equation}
In particular, if $\langle P(y), B\rangle=0$, then $P(\partial)( B\pi_{y,m})=0$.
\end{lem}
\begin{proof}
For $\alpha\in\mathsf{N}_{d,n}$ we compute (as in the proof of \cite[Lemma~19.7]{sch17}),
\[
    \partial^\alpha\pi_{y,m}
    =m(m-1)\dotsm(m+1-n)y^\alpha\pi_{y,m-n}.
\]
Let $P$ be as in \eqref{G19.2}.
By linearity the preceding implies
\[\begin{split}
    P(\partial)F
    &=\sum_{\alpha\in\mathsf{N}_{d,n}}\left\langle \binom{n}{\alpha}A_\alpha,\partial^\alpha F\right\rangle
    =\sum_{\alpha\in\mathsf{N}_{d,n}}\left\langle \binom{n}{\alpha}A_\alpha,\partial^\alpha\left(\sum_{r=1}^s C_r\pi_{\eta_r,m}\right)\right\rangle\\
    &=\sum_{\alpha\in\mathsf{N}_{d,n}}\left\langle \binom{n}{\alpha}A_\alpha,\sum_{r=1}^sC_r\partial^\alpha\pi_{\eta_r,m}\right\rangle\\
    &=\sum_{\alpha\in\mathsf{N}_{d,n}}\left\langle \binom{n}{\alpha}A_\alpha,\sum_{r=1}^sm(m-1)\dotsm(m+1-n)\eta_r^\alpha C_r\pi_{\eta_r,m-n}\right\rangle\\
    &=\sum_{\alpha\in\mathsf{N}_{d,n}}\sum_{r=1}^sm(m-1)\dotsm(m+1-n)\eta_r^\alpha\pi_{\eta_r,m-n}\left\langle \binom{n}{\alpha}A_\alpha,C_r\right\rangle\\
    &=m(m-1)\dotsm(m+1-n)\sum_{r=1}^s\left\langle\sum_{\alpha\in\mathsf{N}_{d,n}} \binom{n}{\alpha}A_\alpha\eta_r^\alpha,C_r\right\rangle\pi_{\eta_r,m-n}\\
    &=m(m-1)\dotsm(m+1-n)\sum_{r=1}^s\left\langle P(\eta_r), C_r\right\rangle\pi_{\eta_r,m-n}.
\end{split}\]
Formula \eqref{G19.15} follows from \eqref{G19.14} by setting $F= B\pi_{y,m}$.
\end{proof}

\begin{lem}\label{pqdiff}
If $P,Q\in\cH_{q;d,m}$, then $[P,Q]=\frac{1}{m!}P(\partial)Q=\frac{1}{m!}Q(\partial)P$.
\end{lem}
\begin{proof}
By the symmetry of the (real-valued) scalar product it is sufficient to prove the first formula.
We first consider the case $P=\binom{m}{\alpha}A_\alpha x^\alpha$ and $Q=\binom{m}{\beta}B_\beta x^\beta$, where $\alpha,\beta\in\mathsf{N}_{d,m}$ and $A_\alpha,B_\beta\in\Hq$, and verify
\begin{equation}\label{G1336}
    \frac{1}{m!}\left\langle\binom{m}{\alpha}A_\alpha,\binom{m}{\beta} B_\beta \partial^\alpha x^\beta\right\rangle
    =\left[\binom{m}{\alpha}A_\alpha x^\alpha,\binom{m}{\beta}B_\beta x^\beta\right].
\end{equation}

First suppose that $\alpha\neq\beta$.
Then $[\binom{m}{\alpha}A_\alpha x^\alpha,\binom{m}{\beta}B_\beta x^\beta]=0$ by \eqref{G19.3}.
Since $\alpha,\beta\in\mathsf{N}_{d,m}$ and $\alpha\neq\beta$, there is an index $j\in\{1,\dotsc,d\}$ such that $\alpha_j>\beta_j$.
Then $\partial_j^{\alpha_j}x_j^{\beta_j}=0$ and hence $\partial^\alpha x^\beta=0$.
Consequently,
\[
    \frac{1}{m!}\left\langle\binom{m}{\alpha}A_\alpha,\binom{m}{\beta}B_\beta \partial^\alpha x^\beta\right\rangle
    =0
    =\left[\binom{m}{\alpha}A_\alpha x^\alpha,\binom{m}{\beta}B_\beta x^\beta\right].
\]

Now assume that $\alpha=\beta$.
Then we have $\alpha_j=\beta_j$ and hence $\partial_j^{\alpha_j}x^{\beta_j}=\alpha_j!$ for all $j\in\{1,\dotsc,d\}$.
This yields $\partial^\alpha x^\beta=\alpha_1!\dotsm\alpha_d!$ and hence $\binom{m}{\beta}B_\beta\partial^\alpha x^\beta=\binom{m}{\alpha}\alpha_1!\dotsm\alpha_d!B_\alpha=m!B_\alpha$.
Consequently,
\[\begin{split}
    \frac{1}{m!}\left\langle\binom{m}{\alpha}A_\alpha,\binom{m}{\beta}B_\beta\partial^\alpha x^\beta\right\rangle
    &=\frac{1}{m!}\left\langle\binom{m}{\alpha}A_\alpha,m!B_\alpha\right\rangle\\
    &=\binom{m}{\alpha}\left\langle A_\alpha,B_\alpha\right\rangle
    =\left[\binom{m}{\alpha}A_\alpha x^\alpha,\binom{m}{\beta}B_\beta x^\beta\right].
\end{split}\]

Hence, \eqref{G1336} holds true for all $\alpha,\beta\in\mathsf{N}_{d,m}$ and all $A_\alpha,B_\beta\in\Hq$.
Consider homogeneous polynomials $P(x)=\sum_{\alpha\in\mathsf{N}_{d,m}}\binom{m}{\alpha} A_\alpha x^\alpha$ and $Q(x)=\sum_{\beta\in\mathsf{N}_{d,m}}\binom{m}{\beta} B_\beta x^\beta$ of $\cH_{q;d,m}$.
Using \eqref{G1336} we have then
\[\begin{split}
    \frac{1}{m!}P(\partial)Q
    &=\frac{1}{m!}\sum_{\alpha\in\mathsf{N}_{d,m}}\left\langle \binom{m}{\alpha}A_\alpha,\partial^\alpha Q\right\rangle\\
    &=\frac{1}{m!}\sum_{\alpha\in\mathsf{N}_{d,m}}\left\langle \binom{m}{\alpha}A_\alpha,\sum_{\beta\in\mathsf{N}_{d,m}}\binom{m}{\beta}B_\beta\partial^\alpha x^\beta\right\rangle\\
    &=\sum_{\alpha,\beta\in\mathsf{N}_{d,m}}\frac{1}{m!}\left\langle \binom{m}{\alpha}A_\alpha,\binom{m}{\beta}B_\beta\partial^\alpha x^\beta\right\rangle\\
    &=\sum_{\alpha,\beta\in\mathsf{N}_{d,m}}\left[\binom{m}{\alpha}A_\alpha x^\alpha,\binom{m}{\beta}B_\beta x^\beta\right]\\
    &=\left[\sum_{\alpha\in\mathsf{N}_{d,m}}\binom{m}{\alpha}A_\alpha x^\alpha,\sum_{\beta\in\mathsf{N}_{d,m}}\binom{m}{\beta}B_\beta x^\beta\right]
    =[P,Q].\qedhere
\end{split}\]
\end{proof}

Combining formula \eqref{G19.14} with Lemma~\ref{pqdiff} we obtain the following:

\begin{cor}\label{C19.9}
If $F:=\sum_{r=1}^s C_r\pi_{\eta_r,m}$ with $\eta_1,\dotsc,\eta_s\in\dR^d$ and $C_1,\dotsc,C_s\in\Hq$, then we have
\begin{equation}\label{G19.16}
    \frac{1}{m!}F(\partial)P
    =[P,F]
    =\sum_{r=1}^s\left\langle P(\eta_r), C_r\right\rangle
    \quad\text{for all }P\in\cH_{q;d,m}.
\end{equation}
\end{cor}

Finally, we turn to the relations between the apolar scalar product and the truncated moment problem.

Let us begin by defining a cone in the vector space $\cH_{q;d,m}$:
\[
    \mathsf{Q}_{q;d,m}
    :=\left\{\sum_{r=1}^s C_r\pi_{\eta_r,m}\colon s\in\dN,\,\eta_1,\dotsc,\eta_s\in\dR^d,\,C_1,\dotsc,C_s\in\Hqgg \right\}.
\]
Clearly,
\[
    \mathsf{Q}_{q;d,m}
    =\cone\left(\{\pi_{y,m}vv^*\colon y\in\dR^d\text{ and }v\in\dC^q\}\right).
\]

\begin{thm}\label{T19.14}
Let $\eta_1,\dotsc,\eta_s\in\dR^d$, $C_1,\dotsc,C_s\in\Hqgg $, and $F:=\sum_{r=1}^s C_r\pi_{\eta_r,m}$.
Then the linear functional $\Gamma_F$ on $\cH_{q;d,m}$ defined by 
\[
    \Gamma_F(P)
    :=\frac{1}{m!}F(\partial)P,
    \quad P\in\cH_{q;d,m},
\]
is a moment functional on the vector space $\cH_{q;d,m}$ on $\cX=\dR^d$.
It satisfies
\begin{equation}\label{G19.28}
    \Gamma_F(P)
    =[P,F]
    =\sum_{r=1}^s\left\langle P(\eta_r), C_r\right\rangle
    \quad\text{for all }P\in\cH_{q;d,m}
\end{equation}
and $\nu:=\sum_{r=1}^s C_r\delta_{\eta_r}$ is a representing measure of $\Gamma_F$.

Each moment functional on $\cH_{q;d,m}$ is of the form $\Gamma_F$ with $F\in\mathsf{Q}_{q;d,m}$ uniquely determined.
\end{thm}
\begin{proof}
Equation \eqref{G19.16} gives \eqref{G19.28}.
The latter means that $\Gamma_F$ is a moment functional on $\cH_{q;d,m}$ and $\nu$ is a representing measure of $\Gamma_F$.

Conversely, let $\Lambda$ be an arbitrary moment functional on $\cH_{q;d,m}$.
If $\Lambda=0$, then $\Lambda=\Gamma_F$ for $F:=O\cdot\pi_{0,m}=O$.
Now let $\Lambda\neq0$.
By \cite[Theorem~5.1]{ms}, $\Lambda$ has an $s$\nobreakdash-atomic representing measure $\nu=\sum_{r=1}^s C_r\delta_{\eta_r}$, where $C_r\in \Hqgg $, $\eta_r\in \dR^d$, and $s\in\dN$.
Then $F:=\sum_{r=1}^s C_r\pi_{\eta_r,m}$ belongs to $\mathsf{Q}_{q;d,m}$ and formula \eqref{G19.16} implies that $\Lambda=\Gamma_F$.

Suppose that $\Gamma_F=\Gamma_G$ for $F,G\in\mathsf{Q}_{q;d,m}$.
Then $[P,F]=[P,G]$ by \eqref{G19.28} and hence $[P,F-G]=0$ for all $P\in\cH_{q;d,m}$.
Because $[\cdot,\cdot]$ is a scalar product, we conclude that $F=G$.
\end{proof}

\begin{exm}\label{E0837}
As usual, $\Delta:=\partial_1^2+\dotsb+\partial_d^2$ denotes the Laplacian.
Let $n\in\dN$, $C\in\Hqgg $, and define $F\colon\dR^d\to\Hq$ by $F(x):=C\lVert x\rVert^{2n}$.

Then the linear functional $\Lambda$ on $\cH_{q;d,2n}$ given by $\Lambda(P):=\Gamma_F(P)=\frac{1}{(2n)!}F(\partial)P$ is a moment functional on $\cH_{q;d,2n}$ and we have
\[
    \Lambda(P)
    =\frac{1}{(2n)!}C\Delta^nP
    =[P,F]
    \quad\text{for all }P\in\cH_{q;d,2n}.
\]

Indeed, since $h_{d,2n}(x):=\lVert x\rVert^{2n}$ belongs to $\mathsf{Q}_{1;d,2n}$ by \cite[Theorem~19.15]{sch17}, $F\in\mathsf{Q}_{q;d,2n}$ and the assertions follow from Theorem~\ref{T19.14}.
\end{exm}


\begin{thebibliography}{Kim14}

\bibitem[Bha07]{bhatia}
R. Bhatia.
\newblock {\em Positive definite matrices}.
\newblock Princeton Series in Applied Mathematics. Princeton University Press,
  Princeton, NJ, 2007.

\bibitem[BW11]{bakony}
M. Bakonyi and H.~J. Woerdeman.
\newblock {\em Matrix completions, moments, and sums of {H}ermitian squares}.
\newblock Princeton University Press, Princeton, NJ, 2011.

\bibitem[CF96]{cf1}
R.~E. Curto and L.~A. Fialkow.
\newblock Solution of the truncated complex moment problem for flat data.
\newblock {\em Mem. Amer. Math. Soc.}, 119(568):x+52, 1996.

\bibitem[CF98]{cf2}
R.~E. Curto and L.~A. Fialkow.
\newblock Flat extensions of positive moment matrices: recursively generated
  relations.
\newblock {\em Mem. Amer. Math. Soc.}, 136(648):x+56, 1998.

\bibitem[CZ13]{cimpric}
J. Cimpri\v{c} and A. Zalar.
\newblock Moment problems for operator polynomials.
\newblock {\em J. Math. Anal. Appl.}, 401(1):307--316, 2013.

\bibitem[Kim14]{Kimsey}
D.~P. Kimsey.
\newblock An operator-valued generalization of {T}chakaloff's theorem.
\newblock {\em J. Funct. Anal.}, 266(3):1170--1184, 2014.

\bibitem[KT22]{KimseyT}
D.~P. Kimsey and M. Trachana.
\newblock On a solution of the multidimensional truncated matrix-valued moment
  problem.
\newblock {\em Milan J. Math.}, 90(1):17--101, 2022.

\bibitem[KW13]{KimseyW}
D.~P. Kimsey and H.~J. Woerdeman.
\newblock The truncated matrix-valued {$K$}-moment problem on {$\mathbb{R}^d$},
  {$\mathbb{C}^d$}, and {$\mathbb{T}^d$}.
\newblock {\em Trans. Amer. Math. Soc.}, 365(10):5393--5430, 2013.

\bibitem[MS16]{mours}
B. Mourrain and K. Schm\"{u}dgen.
\newblock Flat extensions in {$\ast$}-algebras.
\newblock {\em Proc. Amer. Math. Soc.}, 144(11):4873--4885, 2016.

\bibitem[MS23]{ms}
C. M\"adler and K. Schm\"udgen.
\newblock On the truncated matricial moment problem. I, 2023.
\newblock \href{https://arxiv.org/abs/2310.00957}{{\tt arXiv:2310.00957 [math.FA]}}

\bibitem[Puk93]{pukelsheim}
F. Pukelsheim.
\newblock {\em Optimal design of experiments}.
\newblock Wiley Series in Probability and Mathematical Statistics. John Wiley \& Sons, Inc., New York, 1993.

\bibitem[Sch87]{sch87}
K. Schm\"{u}dgen.
\newblock On a generalization of the classical moment problem.
\newblock {\em J. Math. Anal. Appl.}, 125(2):461--470, 1987.

\bibitem[Sch15]{sch16}
K. Schm\"{u}dgen.
\newblock The multi-dimensional truncated moment problem: maximal masses.
\newblock {\em Methods Funct. Anal. Topology}, 21(3):266--281, 2015.

\bibitem[Sch17]{sch17}
K. Schm\"{u}dgen.
\newblock {\em The moment problem}, {\em Graduate Texts in
  Mathematics} vol. 277,
\newblock Springer, Cham, 2017.

\bibitem[TL20]{Le}
C. Trinh~Le.
\newblock Tracial moment problems on hypercubes.
\newblock {\em Oper. Matrices}, 14(4):1015--1027, 2020.

\bibitem[Vas98]{vasilescu}
F.-H. Vasilescu.
\newblock Moment problems for multi-sequences of operators.
\newblock {\em J. Math. Anal. Appl.}, 219(2):246--259, 1998.

\bibitem[Veg00]{vegter}
G. Vegter.
\newblock The apolar bilinear form in geometric modeling.
\newblock {\em Math. Comp.}, 69(230):691--720, 2000.

\bibitem[Wei80]{weid}
J. Weidmann.
\newblock {\em Linear operators in {H}ilbert spaces},  {\em
  Graduate Texts in Mathematics} vol. 68,
\newblock Springer-Verlag, New York-Berlin, 1980.

\end{thebibliography}
\end{document}